\documentclass[a4paper]{amsart}

\usepackage[cp1251]{inputenc}
\usepackage[english]{babel}
\usepackage{amssymb,amsmath,amsthm,color}

\newcommand{\udots}{.\raisebox{.5ex}[0pt][0pt]{.}
\raisebox{1ex}[0pt][0pt]{.}}

\theoremstyle{plain}
\newtheorem{theorem}{Theorem}
\theoremstyle{plain}
\newtheorem{lemma}{Lemma}
\theoremstyle{plain}

\theoremstyle{definition}
\newtheorem{definition} {Definition}
\theoremstyle{definition}
\newtheorem{example} {Example}
\theoremstyle{remark}
\newtheorem*{note}{Remark}

\newcommand{\ord}{\mathop{\mathrm{ord}}}
\newcommand{\Hom}{\mathop{\mathrm{Hom}}}
\newcommand{\divi}{\mathop{\mathrm{div}}}
\newcommand{\rk}{\mathop{\mathrm{rk}}}
\newcommand{\trdeg}{\mathop{\mathrm{tr.deg}}}
\newcommand{\mo}{\mathop{\mathrm{mod}}}

\newcommand{\rr}[2]{\ensuremath{(\{ \alpha_{#1} \},\{\alpha_{#2}\})}}
\newcommand{\rrr}[3]{\ensuremath{ (\{ \alpha_{#1} \}, \{ \alpha_{#2}, \alpha_{#3} \})}}

\newcommand{\e}{\ensuremath{2 \varepsilon}}
\newcommand{\ee}[2]{\ensuremath{\varepsilon_{#1} - \varepsilon_{#2}}}

\newcommand{\eeee}[4]{\ensuremath{\varepsilon_{#1} + \varepsilon_{#2} + \varepsilon_{#3} + \varepsilon_{#4}}}

\definecolor{gray}{gray}{0.8}
\definecolor{black}{gray}{0.0}
\makeatletter
\newcommand{\graybox}[2]{\color{gray}\color@block{#1\unitlength}{#2\unitlength}{0\unitlength}}
\newcommand{\blackbox}[2]{\color{black}\color@block{#1\unitlength}{#2\unitlength}{0\unitlength}}
\makeatother

\newcommand{\graysquare}{\graybox{10}{10}}
\newcommand{\myblacksquare}{\blackbox{10}{10}}
\newcommand{\bs}[3]{\multiput(#1,#2)(10,-10){#3}{\myblacksquare}}
\newcommand{\gs}[3]{\multiput(#1,#2)(-10,0){#3}{\graysquare}}
\newcommand{\gsdown}[3]{\multiput(#1,#2)(0,-10){#3}{\graysquare}}

\begin{document}

\title{Classification of double flag varieties of complexity $0$ and $1$}

\author{Elizaveta Ponomareva}
\address{
Moscow State University\\
Faculty of Mechanics and Mathematics\\
Department of Higher Algebra\\
119991 Moscow, Russia}
\email{lizaveta@yandex.ru}

\date{\today}



\begin{abstract}
A classification of double flag varieties of complexity $0$ and $1$ is obtained. An application of this problem to decomposing tensor products of irreducible representations of semisimple Lie groups is considered.
\end{abstract}

\maketitle

\section {Introduction}

Let $G$ be a semisimple complex Lie group, $B$ a Borel subgroup of $G$. Suppose that the group $G$ acts on an irreducible complex algebraic variety $X$. This action induces an action of $B$ on $X$.

\begin{definition}
The codimension  of a generic $B$-orbit in $X$ is called the complexity of an action $G:X$ and is denoted by $c(X)=c_G(X)$.
\end{definition}

\begin{note}
By the Rosenlicht theorem,
 $c(X) = \trdeg \; \mathbb{C}(X)^B / \mathbb{C}$.
\end{note}

A subgroup $P \subseteq G$  is parabolic if $P$ contains a Borel subgroup. Suppose $P$ and $Q$ are parabolic subgroups. The variety $X = G/P \times G/Q$ is called a double flag variety. This paper is devoted to classification of double flag varieties of complexity at most~$1$. Littelmann \cite{Lit} classified double flag varieties of complexity $0$ for maximal parabolic subgroups. Stembridge \cite{St} classified all double flag varieties of complexity~$0$. Panyushev \cite{Pan} found complexities of double flag varieties for all maximal parabolic subgroups. In this paper we obtain these already known results by a uniform method and complete the classification in the case of complexity~$1$.

The problem of classifying double flag varieties of complexity $0$ and $1$ has an application to decomposing tensor products of irreducible representations of $G$ into irreducible summands. Any irreducible $G$-module can be realized as the space of global sections for some line bundle $\mathcal{L}$ over $G/P$ (here $P$ is parabolic). We may regard the
tensor product of the spaces of sections $H^0(G/P,\mathcal{L}) \otimes H^0(G/Q,\mathcal{M})$ as the space of sections of a line bundle over the product of varieties $G/P$ and $G/Q$, i.e., $H^0(G/P \times G/Q, \mathcal{L}\boxtimes \mathcal{M})$. Here $\mathcal{L}\boxtimes \mathcal{M} \rightarrow G/P \times G/Q$ is a line bundle such that the fibre $(\mathcal{L}\boxtimes \mathcal{M})_{(x, y)}$ over the point $(x,y)$ is the tensor products of the fibers $\mathcal{L}_x$ and $\mathcal{M}_y$ over the points $x \in G/P$ and $y \in G/Q$. If the complexity of the variety $X = G/P \times G/Q$ equals $0$ or $1$, then there exists an effective method to decompose the space of sections $H^0(X,\mathcal{N})$ of a line bundle $\mathcal{N} \rightarrow X$ into irreducible submodules.

 Suppose a semisimple group is decomposed into almost direct product of simple subgroups: $G = G_1 \cdot \ldots \cdot G_s$. Then parabolic subgroups $P, Q \subseteq G$ are decomposed into almost direct products of parabolic subgroups $P_i, Q_i \subseteq G_i$. We have $c_G(G/P \times G/Q) = c_{G_1}(G_1/P_1 \times G_1/Q_1) + \ldots + c_{G_s}(G_s/P_s \times G_s/Q_s)$. So the problem of computing complexity of double flag varieties for semisimple groups reduces to the same problem for simple groups.

 Suppose $G$ is a classical matrix group; then we assume that $B$ consists of upper triangular matrices (here we assume that the group $G$ preserves a bilinear form with an antidiagonal matrix in orthogonal and symplectic  cases).
Then parabolic subgroups containing $B$ have a block-triangular structure and are determined by the sizes of diagonal blocks.  A group $SO_n$ for even $n$ is an  exception. For this group not all parabolic subgroups have this form. The remaining parabolics are transformed to the described form by conjugation with transposition of two middle basic vectors (the diagram automorphism). We mark such parabolic subgroups with strokes.

In case of exceptional groups parabolic subgroups are determined by a subset $\Pi \setminus I$ of the set of simple roots $\Pi$, where $I \subseteq \Pi$ is the system of simple roots of a standard Levi subgroup. Simple roots are numbered as in  \cite{Vin}.

In this paper we prove the following classification theorems.

\begin{theorem} Let $G$ be a classical matrix group ($SL_n, SO_n, Sp_n$). Then all double flag varieties of complexity $0$ and $1$ correspond to the pairs of parabolic subgroups given in Tables \ref{sl}, \ref{so}, \ref{sp} (the classification is given up to permutation of parabolics in a pair in all cases, up to simultaneous transposition with respect to the secondary diagonal for $SL_n$, and up to the diagram automorphism for $SO_{2n}$).
\end{theorem}

\noindent
\begin{table}[!h]
\begin{center}
\begin{tabular}{|c|l|l||l|l|}
\hline & \multicolumn{2}{|c||}{ complexity $0$}  &
 \multicolumn{2}{|c|}{complexity $1$} \\ \hline

number of blocks & & & & \\
 in $P$ and $Q$ & \multicolumn{1}{|c|}{$P$} & \multicolumn{1}{|c||}{$Q$} &  \multicolumn{1}{|c|}{$P$} & \multicolumn{1}{|c|}{$Q$} \\
\hline \hline

2,2 &   $(p_1,p_2)$ &  $(q_1,q_2)$               &&\\ \hline

2,3 &   $(p_1,p_2)$ &  $(1,q_2,q_3)$               &  $(3,p_2)$, $p_2 \geqslant 3$ &  $(q_1,q_2,q_3)$, $q_1,q_2,q_3 \geqslant 2$         \\
    &   $(p_1,p_2)$ &  $(q_1,1,q_3)$               &  $(p_1,p_2)$, $p_1, p_2 \geqslant 3$ &  $(2,2,q_3)$, $q_3 \geqslant 2$   \\
    &   $(2,p_2)$ &  $(q_1,q_2,q_3)$               & $(p_1,p_2)$, $p_1, p_2 \geqslant 3$ &  $(2,q_2,2)$, $q_2 \geqslant 2$    \\ \hline

2,4 &&   & $(2,p_2)$ &  $(q_1,q_2,q_3,q_4)$ \\

    &&   & $(p_1,p_2)$, $p_1, p_2 \geqslant 2$ &  $(1,1,1,q_4)$ \\

    &&   & $(p_1,p_2)$, $p_1, p_2 \geqslant 2$ &  $(1,1,q_3,1)$  \\ \hline

2,$s$ &   $(1,p_2)$ &  $(q_1,q_2, \dots, q_s)$ && \\ \hline

3,3 &&   & $(1,1,p_3)$ &  $(q_1,q_2,q_3)$ \\

    &&   & $(1,p_2,1)$ &  $(q_1,q_2,q_3)$  \\ \hline

\end{tabular}
\caption{pairs of parabolic subgroups corresponding to double flag varieties of complexity $0$ and $1$ for $SL_n$}
\label{sl}
\end{center}
\end{table}

\noindent
\begin{table}[!h]
\begin{center}
\begin{tabular}{|c|l|l||l|l|}
\hline & \multicolumn{2}{|c||}{ complexity $0$}  &
 \multicolumn{2}{|c|}{complexity $1$} \\ \hline

number of blocks & & & & \\
 in $P$ and $Q$ & \multicolumn{1}{|c|}{$P$} & \multicolumn{1}{|c||}{$Q$} &  \multicolumn{1}{|c|}{$P$} & \multicolumn{1}{|c|}{$Q$} \\
\hline \hline

2,2 &   $(p,p)$ &  $(p,p)$               &&                                             \\
    &   $(p,p)$ &  $(p,p)'$             &&                                             \\ \hline
2,3 &   $(p,p)$ & $(q_1,q_2,q_1)$, $q_1 \leqslant 3$  &  $(6,6)$ & $(4,4,4)$                       \\
    &   $(p,p)$ & $(q,2,q)$              &&                                            \\ \hline
2,4 &   $(p,p)$ & $(1,q,q,1)$            &  $(4,4)$ & $(2,2,2,2)$                     \\
    &   $(p,p)$  &  $(1,q,q,1)'$         &  $(5,5)$ & $(2,3,3,2)$                     \\
    &   $(4,4)$  &  $(2,2,2,2)'$         &  $(5,5)$ & $(3,2,2,3)$                     \\
    &&                                   &  $(5,5)$ & $(2,3,3,2)'$\\
    &&                                   &  $(5,5)$ & $(3,2,2,3)'$\\ \hline

2,5 &   $(p,p)$ & $(1,1,q,1,1)$         &  $(4,4)$ & $(1,2,2,2,1)$                \\
    &&                                   &  $(4,4)$ & $(2,1,2,1,2)$                   \\ \hline

2,6 &&                                   &  $(4,4)$ & $(1,1,2,2,1,1)$                 \\
    &&                                   & $(4,4)$  & $(1,1,2,2,1,1)'$                \\ \hline

3,3 &   $(1,p,1)$ & $(q_1,q_2,q_1)$      &   $(2,2,2)$ & $(2,2,2)$                    \\
    &   $(p,1,p)$ &  $(p,1,p)$           &   $(2,p,2)$, $p > 1$ & $(q,1,q)$     \\ \hline

3,4 &   $(1,p,1)$ & $(q_1,q_2,q_2,q_1)$  &   $(2,2,2)$ & $(1,2,2,1)$
\\ \hline

3,5 & &                                  &   $(1,p,1)$ & $(q_1,q_2,q_3,q_2,q_1)$       \\
    & &                                  &   $(2,1,2)$ & $(1,1,1,1,1)$                \\ \hline

3,6 & &                                  &   $(1,p,1)$ &
$(q_1,q_2,q_3,q_3,q_2,q_1)$   \\ \hline

4,4 & &                                  &   $(1,2,2,1)$ & $(1,2,2,1)$                \\
    &&& $(1,2,2,1)$  & $(1,2,2,1)'$  \\ \hline

\end{tabular}
\caption{pairs of parabolic subgroups corresponding to double flag varieties of complexity $0$ and $1$ for $SO_n$}
\label{so}
\end{center}
\end{table}

\noindent
\begin{table}[!h]
\begin{center}
\begin{tabular}{|c|l|l||l|l|}
\hline & \multicolumn{2}{|c||}{complexity $0$}  &
 \multicolumn{2}{|c|}{complexity $1$} \\ \hline

number of blocks & & & & \\
 in $P$ and $Q$ & \multicolumn{1}{|c|}{$P$} & \multicolumn{1}{|c||}{$Q$} &  \multicolumn{1}{|c|}{$P$} & \multicolumn{1}{|c|}{$Q$} \\
\hline \hline

2,2 & $(p,p)$ & $(p,p)$    &&       \\ \hline

2,3 &  $(p,p)$ & $(1,q,1)$ &                         $(p,p)$ & $(2,q,2)$    \\ \hline

2,4 && &                                             $(2,2)$ & $(1,1,1,1)$ \\ \hline

3,3 &  $(1,p,1)$ & $(q_1,q_2,q_1)$  &&     \\  \hline

3,4 && &                                             $(1,p,1)$ & $(q_1,q_2,q_2,q_1)$   \\ \hline

3,5 && &                                             $(1,p,1)$ & $(q_1,q_2,q_3,q_2,q_1)$  \\ \hline

\end{tabular}
\caption{pairs of parabolic subgroups corresponding to double flag varieties of complexity $0$ and $1$ for $Sp_n$}
\label{sp}
\end{center}
\end{table}

\begin{theorem}  1) There are no double flag varieties of complexity $0$ and $1$ for the groups $G_2$, $F_4$ and $E_8$.

 2) For $E_6$, the varieties of complexity $0$ correspond to the following pairs of parabolic subgroups:

\noindent \rr{1}{1}, \rr{1}{2}, \rr{1}{4}, \rr{1}{5}, \rr{1}{6},
\rr{2}{5}, \rr{4}{5}, \rr{5}{5}, \rr{5}{6}, \rrr{1}{1}{5},
\rrr{5}{1}{5};

 the varieties of complexity $1$ correspond to the following pairs of parabolic subgroups:

\noindent \rrr{1}{1}{2}, \rrr{1}{1}{6}, \rrr{1}{4}{5},
\rrr{1}{5}{6}, \rrr{5}{1}{2}, \rrr{5}{1}{6}, \rrr{5}{4}{5}, \rrr{5}{5}{6}.

 3) For $E_7$, the varieties of complexity $0$ correspond to the following pairs of parabolic subgroups:

\noindent \rr{1}{1}, \rr{1}{6}, \rr{1}{7};

 the varieties of complexity $1$ correspond to the following pairs of parabolic subgroups:

\noindent \rr{1}{2}.

\label{special}
\end{theorem}

The paper is organized as follows. In Section $2$, we discuss a method of decomposing the space of sections $H^0(X,\mathcal{N})$ of a line bundle $\mathcal{N} \rightarrow X$ into irreducible submodules whenexer the complexity of $X$ equals $0$ or $1$. Some examples of decomposing tensor products of irreducible representations using this method are considered. In Section $3$, some general theorems concerning complexity of double flag varieties are given. In Sections $4$ and $5$, we obtain the classification of double flag varieties of complexity $0$ and $1$ for classical and exceptional groups, respectively.

\section {Decomposition of spaces of sections}

Let $G$ act on a normal variety $X$. We consider prime $B$-stable divisors on $X$. To each prime divisor $D$ we assign a homomorphism $\ord_D : \mathbb{C}(X)^{\times} \rightarrow \mathbb{Z}$.

Any line bundle over $X$ can be $G$-linearized \cite{KKLV}. Any Cartier divisor $\delta$ is linearly equivalent to a $B$-stable divisor; this can be proved by choosing a $B$-semi-invariant rational section of the line bundle $\mathcal{O}(\delta)$ \cite{FMSS}.

\subsection {Case of complexity $0$}
In this case we have $\mathbb{C}(X)^B = \mathbb{C}$. Therefore any $B$-semi-invariant function is uniquely determined by its weight up to a scalar multiple. The value  $\ord_D(f)$ does not change if we multiply $f$ by a constant. Thus we can map (in general, not injectively) the set of $B$-stable prime divisors  to the group $\Hom(\Lambda, \mathbb{Z})$, where $\Lambda = \Lambda(X)$ is the lattice of eigenweights of $B$-semi-invariant rational functions on $X$. $B$-stable divisors can be regarded as vectors in $\Hom(\Lambda, \mathbb{Z})$; hence $\ord_D f_{\lambda} = \langle v_D, \lambda \rangle$, where $v_D \in \Hom(\Lambda, \mathbb{Z})$ is the  vector corresponding to $D$ and $f_{\lambda}$ is a function of weight $\lambda$.

There are finitely many prime $B$-stable divisors, since they lie in the complement of the open $B$-orbit.

Denote by $V_{\lambda}$ an irreducible $G$-module of highest weight $\lambda$. Denote by $\lambda^{\ast}$ the highest weight of the dual module. A map $\lambda \mapsto \lambda^{\ast}$ can be extended to all weights by linearity.

Now we formulate the main theorem about decomposing spaces of sections.

\begin{theorem}[\cite{Bri}]
Let $X$ be a variety of complexity $0$ and $\delta = \sum m_i D_i$  a Cartier divisor, where $D_i$ are the distinct $B$-stable divisors on $X$. Then

\[ H^0 (X, \mathcal{O}(\delta)) \simeq \bigoplus_{\lambda \in \mathcal{P}(\delta) \cap \Lambda} V_{\lambda + \pi (\delta)}, \]

\noindent where $\pi (\delta)$ is the weight of the canonical section $s_{\delta}$ corresponding to the divisor $\delta$ and
$$\mathcal{P}(\delta) = \{ \lambda \in \Lambda \otimes \mathbb{Q} \mid \langle v_i, \lambda \rangle \geqslant -m_i, \; \forall i \}$$ is a polytope in  $\Lambda \otimes \mathbb{Q}$, where $v_i$ are the vectors corresponding to $D_i$.
\label{decomposition0}
\end{theorem}

\begin{proof}
One of the equivalent definitions of spherical variety is that for any $G$-line bundle $\mathcal{L} \rightarrow X$ the action $G:H^0(X, \mathcal{L})$ is multiplicity-free. (By definition, a variety is spherical if its complexity equals zero.)

Thus it is sufficient to describe the set of highest weights. A $B$-semi-invariant section can be represented as $s = f_{\lambda} s_{\delta}$.
The condition that the divisor  $\divi s = \divi f_{\lambda} + \delta$ is effective is equivalent to  $\lambda \in \mathcal{P}(\delta) \cap \Lambda$.
\end{proof}

\subsection {Case of complexity $1$}

For varieties of complexity $1$ the theory is a bit more complicated.
 Suppose for simplicity that $X$ is a rational variety. Then by the L\"{u}roth theorem we have $\mathbb{C}(X)^B \simeq \mathbb{C}(\mathbb{P}^1)$. Therefore $B$-semi-invariant functions are determined by their weights uniquely up to multiplication by a function from  $\mathbb{C}(X)^B \simeq \mathbb{C}(\mathbb{P}^1)$, i.e., a $B$-semi-invariant function can be represented as $f_{\lambda}q$, where $f_{\lambda}$ is a fixed function of weight $\lambda$ and $q \in \mathbb{C}(\mathbb{P}^1)$. There is a rational map $X \dashrightarrow \mathbb{P}^1$ whose general fibres are the closures of general $B$-orbits. Therefore we can describe prime $B$-stable divisors as follows. Except a finite number of them, prime $B$-stable divisors form a family parameterized by the projective line except a finite number of points.

Similar to the case of complexity $0$ we can associate a vector  $v_D \in \Hom(\Lambda, \mathbb{Z})$ to any $B$-stable prime divisor $D$  restricting $\ord_D$ to $\{ f_{\lambda} \mid \lambda \in \Lambda \}$ (here we assume that the map $\lambda \mapsto f_{\lambda}$ is a group homomorphism). By restricting $\ord_D$ to $\mathbb{C}(X)^B \simeq \mathbb{C}(\mathbb{P}^1)$ we obtain a valuation of $\mathbb{C}(\mathbb{P}^1)$ with center $z_D \in \mathbb{P}^1$ and order $h_D \in \mathbb{Z}_{+}$ of a local coordinate at $z_D$ (if $h_D = 0$, we can take any point from $\mathbb{P}^1$ for $z_D$). Then $\ord_D f = \langle v_D, \lambda \rangle + h_D \ord_{z_D} q$. Thus  we associate with $D$ a triple $(v_D, z_D, h_D)$. Remove sufficiently many points from the projective line. Then we may assume that for remaining points in $\mathbb{P}^1$ there exists a unique prime $B$-stable divisor $D$ such that $z_D = z$, and furthermore, $v_D = 0$, $h_D = 1$.

For varieties of complexity $1$ there is a similar theorem about decomposition of the space of sections:

\begin{theorem}[\cite{Tim}]
Let $X$ be a rational variety of complexity $1$ and  $\delta = \sum m_i D_i$ a Cartier divisor, where the sum ranges over all $B$-stable prime divisors on $X$ (we assume that only finitely many $m_i$ are nonzero). Then

\[ H^0 (X, \mathcal{O}(\delta)) \simeq \bigoplus_{\lambda \in \mathcal{P}(\delta) \cap \Lambda}
m(\delta, \lambda) V_{\lambda + \pi (\delta)}, \]

\noindent where $\pi (\delta)$ is the weight of the canonical section $s_{\delta}$ corresponding to the divisor $\delta$,

\[ \mathcal{P}(\delta) = \{ \lambda \in \Lambda \otimes \mathbb{Q} \mid \langle v_i, \lambda \rangle \geqslant
-m_i, \; \forall i , \; \text{whenever} \; h_i = 0 \}, \]

\noindent where $(v_i, z_i, h_i)$ is the triple corresponding to the divisor $D_i$, and the multiplicity $m(\delta, \lambda)$ of the module $V_{\lambda + \pi (\delta)}$ in the decomposition equals

$$m(\delta, \lambda) = \max \left( 1 + \sum_{z \in \mathbb{P}^1} m_z, 0 \right) , $$

$$\text{where } m_z = \min_{z_i = z, h_i \neq 0} \left[\frac{\langle v_i, \lambda
\rangle + m_i} { h_i }\right], \forall z \in \mathbb{P}^1. $$

\end{theorem}

The proof of this theorem is similar to the proof of Theorem \ref{decomposition0}.

\subsection{Examples}

\begin{example}
Let $G = Sp_n, \; n = 2l$. Consider the double flag variety $X = G/P \times G/Q$ corresponding to the pair $(1,2l-2,1)$, $(l,l)$ of parabolic subgroups. This is a variety of complexity $0$. Suppose $e_1, \dots, e_n$ is the standard basis of $\mathbb{C}^n$,  $\epsilon_i$ are the weights of the vectors $e_i$ with respect to the  diagonal maximal torus $T$, $\omega_i = \epsilon_1 + \dots + \epsilon_i$ are the fundamental weights. Denote by $\ell$ and $S$ a line and an $l$-dimensional subspace corresponding to points of $G/P$ and $G/Q$. Denote by $x_i$ and $y_{i_1,\dots, i_l}$ the Pl\"{u}cker coordinates on $G/P$ and $G/Q$. Denote by  $E_k$ a $B$-stable subspace $\langle e_1, \dots, e_k \rangle$.

Here is a list of $B$-stable prime divisors $D_i$ (determined by geometric conditions on $\ell$, $S$), their equations $F_i$ in Pl\"{u}cker coordinates, degrees and weights of $F_i$:

\noindent
\begin{tabular}{|l|l|l|l|l|}
\hline

$D_1$ & $\ell \subset E_{n-1}$ & $x_n$ & $(1, 0)$ & $\omega_1$ \\ \hline

$D_2$ & $S \cap E_l \neq 0$ & $y_{l+1,\dots,n}$ & $(0,1)$ & $\omega_l$ \\ \hline

$D_3$ & $(S+\ell) \cap E_{l-1} \neq 0$ & $\sum \limits_{i \geqslant l} (-1)^i x_i y_{l,\dots,\hat{i},\dots,n}$  &
$(1,1)$ &  $\omega_{l-1}$ \\ \hline

$D_4$ & $(S+\ell) \cap \ell^{\perp} \cap E_l \neq 0$ &
$\sum \limits_{i \leqslant l} (x_i y_{l+1,\dots,n} + \sum \limits_{j > l} (-1)^{l+j} x_j y_{i,l+1,\dots,\hat{j},\dots,n}) x_{n+1-i}$
& $(2,1)$ & $\omega_l$ \\ \hline

\end{tabular}

\smallskip

The points in the complement of these divisors belong to the open $B$-orbit. Indeed, assume that a point does not belong to $D_2$. Consider  the matrix whose columns are the basis vectors of $S$. By choosing a basis we can assume that the lower $l \times l$ submatrix is the identity matrix.
We can make other matrix entries equal to zero by the action of $B$. Suppose the point in addition does not belong to $D_1$. Then the lowest entry of the column generating $\ell$ is nonzero. Now we can make entries of this column at positions $l+1, \dots, n-1$ equal to zero. Suppose the point in addition does not belong to $D_3$. Then the $l$-th entry of the column generating $\ell$ is nonzero. By the action of $B$ we can make entries at positions $2,\dots,l-1$ equal to zero. Suppose the point does not belong to $D_4$. Then the $1$-st entry in the column is nonzero. By the action of $B$ we can make these three nonzero entries equal to $1$, i.e., now the point has the unique canonical form.

Up to a scalar multiple, $B$-semi-invariant functions are ratios of products of $F_i$ such that degrees in every group of Pl\"{u}cker coordinates for the numerator and the denominator are equal. Hence we can find a lattice $\Lambda(X)$: it is generated by weights $\epsilon_1 - \epsilon_l$ and $\epsilon_1 + \epsilon_l$. We can take $f_{\epsilon_1 - \epsilon_l} = \frac{F_1 F_3}{F_4}$, $f_{\epsilon_1 + \epsilon_l} = \frac{F_1 F_2}{F_3}$ as basis weight functions.
In the basis dual to the weights of these weight functions, the vectors $v_{D_1} = (1,1)$, $v_{D_2} = (0,1)$, $v_{D_3} = (1,-1)$, $v_{D_4} = (-1,0)$ correspond to the divisors.

 Any divisor $\delta$ is equivalent to a linear combination of the preimages of Schubert divisors: $\delta = p D_1 + q D_2$. The space of sections of the line bundle $\mathcal{O}(\delta)$ is the tensor product of the spaces of sections $\mathcal{O}(p \pi_1(D_1))$ and $\mathcal{O}(q \pi_2(D_2))$, where $\pi_1$, $\pi_2$ are projections of $X$ to $G/P$, $G/Q$ and $\pi_1(D_1)$, $\pi_2(D_2)$ are Schubert divisors. The spaces of sections of $\mathcal{O}(p \pi_1(D_1))$ and $\mathcal{O}(q \pi_2(D_2))$ are isomorphic to $V_{p \omega_1}$ and $V_{q \omega_l}$, respectively. Thus for decomposing the product $V_{p \omega_1} \otimes V_{q \omega_l}$ it is sufficient to compute $H^0(G/P \times G/Q, \mathcal{O}(p D_1 + q D_2))$. The weight polytope is equal to $\mathcal{P}(\delta) = \{ \lambda = -a \epsilon_1 - b \epsilon_l \mid 0 \leqslant b \leqslant a \leqslant p, a+b \leqslant 2q \}$. Using Theorem \ref{decomposition0}, we obtain a decomposition:

$$V_{p \omega_1} \otimes V_{q \omega_l} = \bigoplus_{\substack{ 0 \leqslant b \leqslant a \leqslant p \\
a+b \leqslant 2q \\ a\equiv b (\mo 2)}}
 V_{(p+q-a)\epsilon_1 + q\epsilon_2 + \dots + q\epsilon_{l-1} + (q - b)\epsilon_l}.$$

\label{ex0}
\end{example}

\begin{example}
Let $G = SL_n$. Consider the double flag variety corresponding to the pair $(3, p_2)$, $(q_1,q_2,q_3)$ of parabolic subgroups. We assume that $q_1, q_2, q_3 \geqslant 3$. This is a variety of complexity $1$. We use notation similar to notation from Example \ref{ex0}. Assume that $\omega_0 = \omega_n = 0$. Note that $\epsilon_i^{\ast} = -\epsilon_{n+1-i}$.

Denote by $R_i$ and $S_j$ the subspaces corresponding to points in $G/P$ and $G/Q$, the lower index denotes the dimension of a subspace. Here is a list of $B$-stable prime divisors (determined by geometric conditions on $R_i$, $S_j$), degrees of their equations $F_i$ in Pl\"{u}cker coordinates, and weights of $F_i$:

\smallskip

\noindent
\begin{tabular} {|l|l|l|l|}
\hline
$D_1$ & $R_3 \cap E_{n-3} \neq 0$ & $(1, 0, 0)$ & $\omega^{\ast}_3$ \\ \hline

$D_2$ & $S_{q_1} \cap E_{n - q_1} \neq 0$ & $(0, 1, 0)$ & $\omega^{\ast}_{q_1}$ \\ \hline

$D_3$ & $S_{q_1 + q_2} \cap E_{n - (q_1+ q_2)} \neq 0$ & $(0, 0, 1)$ & $\omega^{\ast}_{q_1 + q_2}$\\ \hline

$D_{4,5,6}$ & $\langle R_3 \cap E_{n - 3 + k} + S_{q_1} \cap E_{n - 3 + k} \rangle \cap$ &
$(1, 1, 0)$ & $\omega^{\ast}_{3 - k} + \omega^{\ast}_{q_1 + k}$ \\
& \multicolumn{1}{r|} {$ \cap E_{n - q_1 - k} \neq 0$, \;  $k = 1,2,3$} && \\ \hline

$D_{7,8,9}$ & $\langle R_3 \cap E_{n - 3 + k} + S_{q_1 + q_2} \cap E_{n - 3 + k} \rangle \cap$
& $(1, 0, 1)$ & $\omega^{\ast}_{3 - k} + \omega^{\ast}_{(q_1 + q_2) + k}$\\
& \multicolumn{1}{r|} {$\cap E_{n - (q_1 + q_2) - k} \neq 0$, \;$k = 1,2,3$} && \\ \hline

$D_{10}$ & $\langle \langle R_3 \cap E_{n-1} + S_{q_1} \cap E_{n-1}
\rangle \cap E_{n - q_1 - 1} + $ & $(1, 1, 1)$ &
$\omega^{\ast}_1 + \omega^{\ast}_{q_1 + 1} + \omega^{\ast}_{(q_1 + q_2) + 1}$ \\
& \multicolumn{1}{r|} {$+ S_{q_1 + q_2} \cap E_{n-q_1-1} \rangle \cap E_{n-(q_1 + q_2) - 1} \neq 0$}&& \\ \hline

$D_{11}$ & $\langle \langle R_3 + S_{q_1} \rangle \cap E_{n - q_1 - 2} +
S_{q_1+q_2} \cap$ &
$(1, 1, 1)$ & $\omega^{\ast}_{q_1+ 2} + \omega^{\ast}_{(q_1+q_2) + 1}$ \\
& \multicolumn{1}{r|} {$\cap E_{n-q_1-2} \rangle \cap E_{n-(q_1+q_2)-1} \neq 0$} && \\ \hline

$D_{12}$ & $\langle \langle R_3 + S_{q_1} \rangle \cap E_{n - q_1 - 1} +
S_{q_1+q_2} \cap $ &
$(1, 1, 1)$ & $\omega^{\ast}_{q_1+ 1} + \omega^{\ast}_{(q_1+q_2) + 2}$\\
& \multicolumn{1}{r|} {$\cap E_{n-q_1-1} \rangle \cap E_{n-(q_1+q_2)-2} \neq 0$} && \\ \hline

$D_{13}$ & $\langle R_3 \cap E_{n-2} + \langle
R_3 + S_{q_1} \rangle \cap E_{n - q_1 - 2} + $ &
$(2, 1, 1)$ & $\omega^{\ast}_2 + \omega^{\ast}_{q_1+2} + \omega^{\ast}_{(q_1+q_2) + 2}$\\
& \multicolumn{1}{r|} {$+ S_{q_1+q_2} \cap E_{n-q_1-2} \rangle \cap E_{n-(q_1+q_2)-2} \neq 0$} && \\ \hline

$D(z)$ & $F_4 F_8 F_{11} - z F_5 F_7 F_{12}$  & $(3,2,2)$ & $\omega_1^{\ast} + \omega_2^{\ast} + \omega_{q_1+1}^{\ast} + \omega_{q_1+2}^{\ast} + $ \\
$z \in \mathbb{P}^1$
&&& \multicolumn{1}{r|}{$+  \omega_{q_1+q_2 +1}^{\ast} + \omega_{q_1+q_2+2}^{\ast}  $ } \\
$z \neq 0,1,\infty$ &&& \\ \hline

\end{tabular}

\medskip

Consider polynomials in $F_i$, $i = 1,\dots 13$, as polynomials in Pl\"{u}cker coordinates. The subspace of polynomials of weight $\omega_1^{\ast} + \omega_2^{\ast} + \omega_{q_1+1}^{\ast} + \omega_{q_1+2}^{\ast} + \omega_{q_1+q_2 +1}^{\ast} + \omega_{q_1+q_2+2}^{\ast}$ and multidegree $(3,2,2)$ has dimension $2$ and is linearly spanned by polynomials $F_4 F_8 F_{11}$, $F_5 F_7 F_{12}$, $F_{10} F_{13}$. Weight subspaces of lower degrees have dimension $1$. These three polynomials are linearly dependent. Multiplying $F_i$ by a scalar we may assume that the equation  of linear dependence is $F_4 F_8 F_{11} - F_5 F_7 F_{12} + F_{10} F_{13} = 0$.
We can regard the polynomials in this weight subspace as linear forms on $\mathbb{P}^1$ by taking $F_4 F_8 F_{11}$ and $F_5 F_7 F_{12}$ for homogeneous coordinates.

 The valuation corresponding to a divisor $D(z)$ has order $h=1$ and center at $z$. The vector $v_{D(z)}$ corresponding to this divisor equals zero.
The valuations corresponding to the divisors $D_4$, $D_8$ and $D_{11}$ have order $h = 1$ and center at $0$; the valuations corresponding to the divisors $D_5$, $D_7$ and $D_{12}$ have order $h = 1$ and center at $\infty$; the valuations corresponding to the divisors $D_{10}$ and $D_{13}$ have order $h = 1$ and center at $1$. For other divisors $D_i$ the corresponding valuations have order $h = 0$.

$B$-semi-invariant functions are constructed in the same way as in Example \ref{ex0}, but up to multiplication by a function from $\mathbb{C}(X)^B$. $B$-invariant functions are ratios of homogeneous polynomials of same degree in coordinates $F_4 F_8 F_{11}$ and $F_5 F_7 F_{12}$, i.e., the field $\mathbb{C}(X)^B$ is generated by the function $\frac{F_4 F_8 F_{11}}{F_5 F_7 F_{12}}$.

 The lattice $\Lambda(X)$ is generated by the weights $\epsilon_i - \epsilon_j$, where $i$ and $j$ are numbers from different triples  $(1,2,3)$, $(q_1+1,q_1+2,q_1+3)$, and $(q_1+q_2+1,q_1+q_2+2,q_1+q_2+3)$.

 We take the following functions as the basis weight functions:

 \begin{gather*}
\frac{F_4 F_{11}}{F_5 F_{10}} = f_{\epsilon^{\ast}_2 - \epsilon^{\ast}_1},
 \frac{F_1 F_{11}}{F_5 F_7} = f_{\epsilon^{\ast}_3 - \epsilon^{\ast}_1},
 \frac{F_{12}}{F_2 F_8} = f_{\epsilon^{\ast}_{q_1+1} - \epsilon^{\ast}_1},
 \frac{F_{11}}{F_{10}} = f_{\epsilon^{\ast}_{q_1+2} - \epsilon^{\ast}_1},
 \frac{F_6}{F_5} = f_{\epsilon^{\ast}_{q_1+3} - \epsilon^{\ast}_1}, \\
 \frac{F_{11}}{F_3 F_5} = f_{\epsilon^{\ast}_{q_1+q_2+1} - \epsilon^{\ast}_1},
 \frac{F_{12}}{F_{10}} = f_{\epsilon^{\ast}_{q_1+q_2+2} - \epsilon^{\ast}_1},
 \frac{F_9}{F_8} = f_{\epsilon^{\ast}_{q_1+q_2+3} - \epsilon^{\ast}_1}.
 \end{gather*}

\smallskip

Let $a_i$ be the coordinates of $\lambda \in \Lambda \otimes \mathbb{Q}$ in the basis of weights of the above $B$-semi-invariant functions. Let $\delta = m_1 D_1 + m_2 D_2 + m_3 D_3$. Then we obtain the following inequalities on coordinates defining the polytope $\mathcal{P}(\delta)$:

\noindent
 $$a_2 \geqslant -m_1, a_3 \leqslant m_2, a_6 \leqslant m_3, a_5 \geqslant 0, a_8 \geqslant 0,$$

\noindent
and the following decomposition:

\[   V_{m_1 \omega_3} \otimes V_{m_2 \omega_{q_1} + m_3
\omega_{q_1+q_2}} = \bigoplus m(\bar{a})V_{\lambda(\bar{a}, \bar{m})} ,\]

\noindent
where

\begin{multline*}
m(\bar{a}) = \max(0, 1 + \min(-a_1 -a_2 -a_5 -a_6, -a_2, a_3 + a_7) + \\
+ \min(a_1, -a_3 -a_8, a_1 + a_2 + a_4 +a_6) + \min(-a_1 - a_4 -a_7, 0)),
\end{multline*}

\noindent
\begin{multline*} \lambda(\bar{a}, \bar{m}) =
m_1 \omega_3 + m_2 \omega_{q_1} + m_3 \omega_{q_1+q_2} - (a_1 +
\dots + a_8) \epsilon_1 + a_1 \epsilon_2 + a_2 \epsilon_3 + \\ +
 a_3 \epsilon_{q_1+1} + a_4 \epsilon_{q_1+2} + a_5 \epsilon_{q_1+3} +
a_6 \epsilon_{q_1+q_2+1} + a_7 \epsilon_{q_1+q_2+2} + a_8 \epsilon_{q_1+q_2+3}, \end{multline*}

\noindent
and the sum ranges over all $a_i$ that satisfy the inequalities given above.

\end{example}

\section{Some theorems about complexity of double flag varieties}
Now we formulate some theorems. The theorem we need for computing the complexity of a double flag variety is due to Panyushev:

\begin{theorem}[\cite{Pan}]
 Suppose $P$ and $Q$ are decomposed into semidirect product of the standard Levi subgroup and the unipotent radical: $P = L \rightthreetimes P_u$, $Q = M \rightthreetimes Q_u$. Then the complexity of the action $G : G/P \times G/Q$ equals the complexity of the action  $L \cap M : \mathfrak{p}_u \cap \mathfrak{q}_u$, where $\mathfrak{p}_u$ and $\mathfrak{q}_u$ are the Lie algebras of  $P_u$ and $Q_u$.
\label{theorm_panushev}
\end{theorem}

\begin{lemma}
 The complexity does not change if we swap $P$ and $Q$.
\end{lemma}

\begin{lemma}
Let $P' \subseteq P$, $Q' \subseteq Q$ be parabolic subgroups. Then the complexity of the double flag variety for the pair $(P', Q')$ is not less than for the pair $(P, Q)$.
\label{subgroup}
\end{lemma}

\begin{proof}  There exists a $G$-equivariant surjective morphism $G/P' \times G/Q' \rightarrow G/P \times G/Q$. So the codimension of a general $B$-orbit on $G/P \times G/Q$ is not greater that the corresponding codimension on $G/P' \times G/Q'$.
\end{proof}

\begin{lemma}
$ c \geqslant \frac{1}{2}(\dim G - \dim L - \dim M - \dim T) $, where $c$ is the complexity of the action.
\label{complexity}
\end{lemma}

\begin{proof}
It is easy to see that $c \geqslant  \dim(\mathfrak{p}_u \cap \mathfrak{q}_u) - \dim(L \cap M \cap B)$. Besides, we have $\dim(L \cap M \cap B) = \frac{1}{2}(\dim (L \cap M) + \dim T)$, $\dim(\mathfrak{p}_u \cap \mathfrak{q}_u) = \frac{1}{2}(\dim G - \dim L - \dim M + \dim(L \cap M)$. Substituting these equalities in the first inequality, we prove the lemma.
\end{proof}

\section{Case of classical matrix groups}

In this section $G$ denotes $SL_n$, $SO_n$, or $Sp_n$. We assume that a Borel subgroup $B \subseteq G$ consists of upper-triangular matrices, that $SO_n$ preserves the quadratic form with the matrix

$$\begin{pmatrix}
 \raisebox{-0.7em}{\parbox[t][0.1\height][t]{0.5em}{\Huge{0}}} && 1 \\
  & \udots &
  \\ 1 && \parbox[b][0.1\height]{0.7em}{\Huge{0}}
\end{pmatrix},$$

\noindent
and that $Sp_n$ preserves the skew-symmetric bilinear form with the matrix

$$\begin{pmatrix}
\raisebox{-1.5em}{\parbox[t][0.1\height][t]{0.5em}{\Huge{0}}} &&&&& 1 \\
&&&& \udots  & \\
&&& 1 && \\
&& -1 &&& \\
& \udots &&&& \\
-1 &&&&& \raisebox{0.5em}{\parbox[b][0.1\height][b]{0.9em}{\Huge{0}}}  \\
\end{pmatrix}.$$

For computing the complexity we use Theorem~\ref{theorm_panushev}. Now we describe Levi subgroups, Lie algebras of unipotent radicals and their intersections.

The Levi subgroup $L$ (or $M$) consists of block-diagonal matrices; for groups $SO_n$ and $Sp_n$ the sizes of these blocks are symmetric with respect to the secondary diagonal and the matrices standing at symmetric places are $A$ and ${(A^{S})}^{-1}$ (here $S$ denotes the transposition with respect to the secondary diagonal) and the central block (it exists if the number of blocks is odd) is an orthogonal or symplectic matrix respectively.

The Lie algebra $\mathfrak{sl}_n$ consists of matrices with trace $0$; the Lie algebra $\mathfrak{so}_n$ in the chosen  basis consists of matrices which are antisymmetric with respect to the secondary  diagonal; the Lie algebra $\mathfrak{sp}_n$ in the chosen basis consists of the following matrices: divide a matrix into $4$ equal square parts, then the upper right part and the lower left part are symmetric with respect to the secondary diagonal, the other two parts are antisymmetric to each other with respect to the secondary diagonal. Matrices in the Lie algebra of the unipotent radical $\mathfrak {p}_u$ (or $\mathfrak {q}_u$) have zeroes below the diagonal and in diagonal blocks.

For $SO_n$ with even $n$ there exists another class of parabolic subgroups (we call them \emph{special}). We can obtain these subgroups from block-triangular parabolic subgroups without central diagonal block by conjugation with transposition of two middle basis vectors. We consider special parabolic subgroups separately.

Matrices from $L \cap M$ consist of several diagonal square blocks. We denote these blocks by $A_1, \dots ,A_r$ and their sizes by $k_1, \dots ,k_r$. Besides, for $SO_n$ and $Sp_n$ we have a relation $k_i=k_{r+1-i}$. Note that for  $SO_n$ a middle pair of blocks of sizes $1$ and $1$ is the same as one middle block of size $2$. Further, we assume that parabolic subgroups are not special. We consider the case of special subgroups  separately. Matrices from $\mathfrak{p}_u \cap \mathfrak{q}_u$ consist of submatrices $X_{ij}$, where the matrix $X_{ij}$ is of size $k_i \times k_j$ and stands at the intersection of rows passing through  $A_i$ and columns passing through $A_j$. Besides, $X_{ij}=0$ if $i \geqslant  j$ or if there exists a matrix from $L$ or $M$ with nonzero entries at the place of $X_{ij}$. By ``blocks'' we often mean \textsl{nonzero} matrices $X_{ij}$. A Borel subgroup in $L \cap M$ is $L \cap M \cap B$, i.e., the intersection of $L \cap M$ with upper triangular matrices. The group $L \cap M \cap B$ acts on $\mathfrak{p}_u \cap \mathfrak{q}_u$ by conjugation; matrices $X_{ij}$ are transformed to $A_i X_{ij} A_j^{-1}$.

The idea is to consider all possible locations of blocks $X_{ij}$ and to compute complexity for all sizes of blocks for each location. We need the following lemmas to simplify the case-by-case considerations and  to reduce the number of possible cases.

\begin{lemma}

The complexity does not change if we transpose simultaneously $P$ and $Q$ with respect to the secondary diagonal.
\end{lemma}

\begin{note}
This lemma gives simplification only for $SL_n$.
\end{note}

\begin{lemma}
Consider an action, obtained from the original action by one of the following operations (or their combination):

\begin{itemize}
\item remove some blocks  $X_{ij}$ (i.e., we assume that some $X_{ij}$ are equal to $0$),

\item remove some matrices $A_i$ and blocks $X_{ij}$ in corresponding rows and columns.
\end{itemize}

\noindent
Then the complexity for the new action  cannot be greater than for the original one.
In other words, we consider only ``a part of an action''.
\label{subblocks}
\end{lemma}

\begin{proof}
The first operation corresponds to the restriction of an action to a $G$-stable subspace. The complexity of the action on a $G$-stable subvariety cannot be greater than the complexity of the action on the initial variety \cite{complexity_subvariety}.

The second operation corresponds to considering a quotient representation for which the complexity can be only less or equal than the original complexity.
\end{proof}

\begin{lemma}
Suppose there are $4$ nonzero matrices $X_{pq}$ standing at vertices of a rectangle, i.e., they have indices $ij$, $il$, $kj$ and $kl$. We require that these matrices do not stand on the secondary diagonal for $SO_n$. Then there is a rational invariant for the action of  $B \cap L \cap M$. We call this invariant the invariant of type ``square''.
\label{square}
\end{lemma}

\begin{note} The matrices from $\mathfrak{so}_n$ have zeroes on the secondary diagonal. That is the reason why we have  additional restriction on the positions of blocks for $SO_n$.
\end{note}

\begin{proof} Suppose $a_i$, $a_k$ are right lower entries of matrices $A_i$, $A_k$, $a_j$, $a_l$ are left upper entries of matrices  $A_j$, $A_l$, and $x_{ij}$, $x_{il}$, $x_{kj}$, $x_{kl}$ are left lower entries of matrices $X_{ij}$, $X_{il}$, $X_{kj}$, $X_{kl}$. Then $x_{pq} \to a_p x_{pq} a_q^{-1} $, $p=i,k$, $q=j,l$. It is easy to see that $x_{ij} x_{kj}^{-1} x_{kl} x_{il}^{-1}$ is an invariant.
\end{proof}

\begin{lemma}
  Suppose there are $3$ nonzero matrices $X_{pq}$ standing in a special way at vertices of a rectangular triangle, i.e., they have indices  $ij$, $ik$, $jk$. We require that these matrices do not stand on the secondary diagonal for $SO_n$.  Then there is a rational invariant for the action of $B \cap L \cap M$. We call this invariant the invariant of type ``triangle''.
\label{triangle}
\end{lemma}

\begin{proof}  Suppose $\bar{x}_{ij}$ is the lowest row of the matrix $X_{ij}$, $x_{ik}$ is the left lowest entry of  $X_{ik}$, and $\bar{x}_{jk}$ is the left column of $X_{jk}$. It is easy to check that
$\frac { \bar{x}_{ij} \cdot \bar{x}_{jk} } {x_{ik} }$ is an invariant.
\end{proof}

\smallskip

\begin{note} The invariants of type ``square'' and ``triangle'' do not change for groups $SO_n$ and $Sp_n$ if we consider other blocks obtained by  transposition with respect to the secondary diagonal.
\end{note}

\begin{lemma}
 Suppose there are $3$ nonzero matrices  $X_{ij}$ in one row such that their height is at least $2$.  We require that these matrices do not stand on the secondary diagonal for group $SO_n$. Then the complexity is at least $1$. If there are $4$ such matrices, then the complexity is at least $2$.
\label{in_one_row_3_4}
\end{lemma}

 \begin{proof}
 We prove the first statement for $SL_n$ (since $SO_n$ and $Sp_n$ are subgroups of $SL_n$, we obtain this statement  for other groups as a consequence). By the group action we can make left lower entry and the entry above of the first general matrix equal to $1$ and $0$ respectively. We do not want to change these entries further. Thus we can act on the left only by multiplication by a matrix such that its lower right $2 \times 2$ submatrix is diagonal. Then we can make the same entries of the second general matrix equal to $1$. In order not to change these $4$ entries we must act on the left only by matrices having $\lambda E$ as the lower right $2 \times 2$ submatrix. Now consider the same two entries of the third matrix (they are nonzero for a general matrix): they are multiplied by one and the same number. We can make one of them equal to $1$ and after that we cannot change  another one without changing other $5$ considered entries. Thus general orbits depend at least on one continuous parameter, i.e., $c(X) \geqslant 1$. The proof for $4$ matrices is similar.
\end{proof}

Now we discuss a method for computing complexity of the action of $L \cap M$ on $\mathfrak{p}_u \cap \mathfrak{q}_u$. The Lie algebras of $SO_n$ and $Sp_n$ have symmetry in their block structure. So it is sufficient to consider the blocks on and below the secondary  diagonal.
By the action of $B \cap L \cap M$ we can put our blocks, one by one, in some canonical form and consider the action of the stabilizer of this canonical form on the remaining blocks. The number of the parameters left is the complexity. The same method was used in the proof of Lemma \ref{in_one_row_3_4}.

For $SO_n$ with even $n$ there are special parabolic subgroups, which we mark with strokes. We may assume that only one of the parabolic subgroups is special and the second one does not have a middle block (in the converse case we apply the automorphism of $SO_n$ that transposes two middle basis vectors).
We can estimate the complexity from below by the complexity of another action such that  both parabolic subgroups are not special. For this, let us conjugate the special subgroup with transposition of two middle basis vectors and replace two middle blocks by one (here we may assume that the size of two middle blocks is not greater than the respective size for the second group). We enlarge the parabolic subgroup, so the complexity can only become smaller.

 Now consider particular cases. The pictures show the location of nonzero blocks $X_{ij}$ and matrices $A_i$; the blocks $X_{ij}$ are grey and the matrices $A_i$ are black. We denote the complexity  by $c$. We enumerate the possible locations of the blocks $X_{ij}$. If Lemmas \ref{subblocks}, \ref{square}, \ref{triangle}, \ref{in_one_row_3_4} give an estimate $c \geqslant 2$ for a given case, we shall not consider this case. We shall not consider cases, that are symmetrical to the cases already considered. The results of our considerations are presented below. We indicate only the cases in which the complexity is not greater than~$1$.

\subsection{Group $\mathbf{SL_n}$}

\noindent \raisebox{1pt}{1.}
\begin{picture}(25,25)(0,10)
\gs{10}{10}{1}
\bs{0}{10}{2}
\put(10,10){\line(1,0){10}} \put(10,20){\line(1,0){10}}
\put(10,10){\line(0,1){10}} \put(20,10){\line(0,1){10}}
\end{picture}
\quad
\begin{picture}(35,35)(0,20)
\gs{20}{20}{1}
\bs{0}{20}{3}
\put(30,30){\line(-1,0){10}} \put(30,20){\line(-1,0){10}}
\put(30,30){\line(0,-1){10}} \put(20,30){\line(0,-1){10}}
\end{picture}
\quad \raisebox{0pt}{
\parbox[t]{5cm}{

$k_1, k_r$ arbitrary \quad  $c=0$ }}

\bigskip


\noindent \raisebox{21pt}{2.}
\begin{picture}(35,35)
\gs{20}{20}{2}
\bs{0}{20}{3}
\put(30,30){\line(-1,0){20}} \put(30,20){\line(-1,0){20}}
\put(30,30){\line(0,-1){10}} \put(20,30){\line(0,-1){10}}
\put(10,30){\line(0,-1){10}}
\end{picture}
\quad
\begin{picture}(45,45)(0,10)
\gs{30}{30}{2}
\bs{0}{30}{4}
\put(40,40){\line(-1,0){20}} \put(40,30){\line(-1,0){20}}
\put(40,40){\line(0,-1){10}} \put(30,40){\line(0,-1){10}}
\put(20,40){\line(0,-1){10}}
\end{picture}
\quad \raisebox{30pt}{
\parbox[t]{5cm}{

\begin{tabbing}
$k_1 \leqslant 2$ \qquad \qquad \quad \= $k_{r-1}, k_r$ arbitrary \qquad \=  $c=0$ \\
 $k_1$ arbitrary \> $k_{r-1} = 1$ or $k_r = 1$   \> $c=0$  \\
 $k_1 \geqslant 3$ \>  $k_{r-1}=k_r=2$              \>   $c=1$   \\
 $k_1 = 3$ \>  $k_{r-1}, k_r \geqslant 2$           \>  $c=1$   \\
\end{tabbing}

}}

 As an example, we consider this case in details. Without loss of generality we may assume that $k_{r-1} \geqslant k_r$. Acting from the right, we put two general blocks in the following form: the entries on the secondary diagonal coming from the lower left corner are equal to $1$ and the entries to the right of this diagonal are equal to $0$. We can make entries of the first block above the secondary diagonal coming from the left lower corner equal to $0$. Let us find the stabilizer of this form. $A_1$ has zeroes in all entries except the diagonal and the left upper $\max(k_1 - k_{r-1}, 0) \times \max(k_1 - k_{r-1}, 0)$ submatrix. $A_{r-1}$ and $A_r$ have similar form, but the submatrix is lower right and of the size $\max(k_i - k_1, 0) \times  \max(k_i - k_1, 0)$, where $i = r-1, r$, respectively; the diagonal entries are equal to the diagonal entries of $A_1$ (in order to preserve $1$'s in blocks). Now we can make entries in the first column and rows $2, \dots, \min(k_1, k_{r-1})$ from the bottom of the second block equal to $1$. Consider the stabilizer of this form. All diagonal entries of $A_1$, $A_{r-1}$ and $A_r$ except entries in the considered submatrices are equal to one and the same number $\lambda$.

Suppose $k_1 \leqslant 2$ or $k_r = 1$. Then we can make entries in the  first column of the second block equal to $1$. Thus there are no free parameters, i.e., a general point lies in the orbit of the point of the described form, therefore $c = 0$.

Suppose $k_1 = 3$, $k_{r-1} \geqslant k_r \geqslant 2$. Then we can make the entry in the first column and in the third row from the bottom of the second block equal to $1$ (if it is not already $1$). Then $A_1 = \lambda E$ and the two upper diagonal entries of $A_r$ equal $\lambda$. Thus we cannot change the entry in the second column and in the third row from the bottom. Thus a general orbit depends on one continuous parameter, i.e., $c = 1$.

Now suppose $k_{r-1} = k_r = 2$, $k_1 \geqslant 4$. Consider the submatrix of the second block above the second row from the bottom. We can multiply it on the left by any upper triangular matrix and the action on the right reduces to multiplication of all entries by one and the same number. We can make the lower $2 \times 2$ submatrix of this submatrix equal to
$ \bigl(\begin{smallmatrix} 0 & 1 \\ 1 & \ast \\ \end{smallmatrix}\bigr)$ and all entries above equal to zero. We cannot change the entry $\ast$, hence $c = 1$.

It remains to consider the case $k_1 \geqslant 4$, $k_{r-1} \geqslant 3$, $k_r \geqslant 2$. We can make the entry in the first column and in the fourth row from the bottom of the second block equal to $1$ (if it is not already $1$). Then the lower $4 \times 4$ submatrix of $A_1$ equals $\lambda E$. Consider the entries in the second column and rows $3$ and $4$ from the bottom of the second block: we cannot change them.  Thus $c \geqslant 2$.


\noindent \raisebox{31pt}{3.}
\begin{picture}(45,45)
\gs{30}{30}{3}
\bs{0}{30}{4}
\put(40,40){\line(-1,0){30}} \put(40,30){\line(-1,0){30}}
\put(40,40){\line(0,-1){10}} \put(30,40){\line(0,-1){10}}
\put(20,40){\line(0,-1){10}} \put(10,40){\line(0,-1){10}}
\end{picture}
\quad
\begin{picture}(55,55)(0,10)
\gs{40}{40}{3}
\bs{0}{40}{5}

\put(50,50){\line(-1,0){30}} \put(50,40){\line(-1,0){30}}
\put(50,50){\line(0,-1){10}} \put(40,50){\line(0,-1){10}}
\put(30,50){\line(0,-1){10}} \put(20,50){\line(0,-1){10}}
\end{picture}
\quad \raisebox{40pt}{
\parbox[t]{5cm}{

\begin{tabbing}
$k_1=1$ \quad \= $k_{r-2}, k_{r-1}, k_r$ arbitrary \qquad \qquad  \= $c=0$ \\
$k_1=2$       \>  $k_{r-2}, k_{r-1}, k_r$ arbitrary \> $c=1$ \\
 $k_1 \geqslant 3$ \>  $k_{r-2}=k_{r-1}=k_r=1$         \>  $c=1$   \\

\end{tabbing}

}}

\bigskip

\noindent If we add blocks to the first row, then their height cannot be greater than $1$ by Lemma~\ref{in_one_row_3_4}.


\noindent \raisebox{46pt}{4.}
\begin{picture}(60,60)
\gs{20}{45}{2}
\gs{45}{45}{1}
\bs{0}{45}{2}
\bs{45}{0}{1}
\put(55,55) {\line(-1,0){10}} \put(30,55) {\line(-1,0){20}}
\put(55,45) {\line(-1,0){10}} \put(30,45) {\line(-1,0){20}}
\put(55,55) {\line(0,-1){10}} \put(45,55) {\line(0,-1){10}}
\put(30,55) {\line(0,-1){10}} \put(20,55) {\line(0,-1){10}}
\put(10,55) {\line(0,-1){10}}

\multiput(22,33)(5,-5){5}{\circle*{1}}
\multiput(32,47)(5,0){3}{\circle*{1}}
\end{picture}
\quad
\begin{picture}(70,70)(0,10)
\gs{30}{55}{2}
\gs{55}{55}{1}
\bs{0}{55}{2}
\bs{55}{0}{1}
\put(65,65) {\line(-1,0){10}} \put(40,65) {\line(-1,0){20}}
\put(65,55) {\line(-1,0){10}} \put(40,55) {\line(-1,0){20}}
\put(65,65) {\line(0,-1){10}} \put(55,65) {\line(0,-1){10}}
\put(40,65) {\line(0,-1){10}} \put(30,65) {\line(0,-1){10}}
\put(20,65) {\line(0,-1){10}}
\multiput(22,43)(5,-5){7}{\circle*{1}}
\multiput(42,57)(5,0){3}{\circle*{1}}
\end{picture}
\quad \raisebox{45pt}{
\parbox[t]{5cm}{
$k_1 = 1$ \quad $c=0$ }}

\medskip

\noindent \raisebox{21pt}{5a.}
\begin{picture}(35,35)
\gs{20}{20}{2}
\gs{20}{10}{1}
\bs{0}{20}{3}
\put(30,30){\line(-1,0){20}} \put(30,20){\line(-1,0){20}}
\put(30,10){\line(-1,0){10}} \put(30,30){\line(0,-1){20}}
\put(20,30){\line(0,-1){20}} \put(10,30){\line(0,-1){10}}
\end{picture}
\quad \raisebox{30pt}{
\parbox[t]{5cm}{
\begin{tabbing}

at least two of $k_1, k_2, k_3$ equal $1$ \quad \= $c = 1$ \\

\end{tabbing}

}}

\noindent \raisebox{31pt}{5b.}
\begin{picture}(45,45)
\gs{30}{30}{2}
\gs{30}{20}{1}
\bs{0}{30}{4}
\put(40,40){\line(-1,0){20}} \put(40,30){\line(-1,0){20}}
\put(40,20){\line(-1,0){10}} \put(40,40){\line(0,-1){20}}
\put(30,40){\line(0,-1){20}} \put(20,40){\line(0,-1){10}}
\end{picture}
\quad \raisebox{40pt}{
\parbox[t]{5cm}{
\begin{tabbing}

 $k_1 = 1$ or $k_4 = 1$ \quad   \= $k_2, k_3$ arbitrary \qquad \quad   \= $c = 0$  \\
 $k_1 = k_4 = 2$   \> $k_2, k_3$ arbitrary \> $c = 1$ \\
 $k_1 = 2$, $k_4 \geqslant 3$   \> $k_2 = 1$, $k_3$ arbitrary  \> $c = 1$ \\
 $k_1 \geqslant 3$, $k_4 = 2$   \> $k_2$ arbitrary, $k_3 = 1$ \> $c = 1$ \\
\end{tabbing}
}}


\noindent \raisebox{31pt}{6.}
\begin{picture}(45,45)
\gs{30}{30}{3}
\gs{30}{20}{1}
\bs{0}{30}{4}
\put(40,40){\line(-1,0){30}} \put(40,30){\line(-1,0){30}}
\put(40,20){\line(-1,0){10}} \put(40,40){\line(0,-1){20}}
\put(30,40){\line(0,-1){20}} \put(20,40){\line(0,-1){10}}
\put(10,40){\line(0,-1){10}}
\end{picture}
\quad \raisebox{40pt}{
\parbox[t]{5cm}{
\begin{tabbing}

 $k_1 = 1$ \qquad \= $k_2=1$ or $k_4=1$ \quad \= $c=1$ \\

\end{tabbing}
}}

\smallskip


\noindent \raisebox{11pt}{7.}
\begin{picture}(50,25)
\gs{0}{10}{1}
\gs{35}{10}{2}
\gs{35}{0}{1}
\put(45,20){\line(0,-1){20}} \put(35,20){\line(0,-1){20}}
\put(25,20){\line(0,-1){10}} \put(10,20){\line(0,-1){10}}
\put(0,20){\line(0,-1){10}} \put(45,20){\line(-1,0){20}}
\put(45,10){\line(-1,0){20}} \put(45,0){\line(-1,0){10}}
\put(10,20){\line(-1,0){10}} \put(10,10){\line(-1,0){10}}
\multiput(12,12)(5,0){3}{\circle*{1}}
\end{picture}

\noindent
This case appears only when the number of blocks in the first row is $s \leqslant 3$. This is Case $5$ or $6$.


\noindent \raisebox{31pt}{8.}
\begin{picture}(45,45)
\gs{30}{30}{2}
\gs{30}{20}{2}
\bs{0}{30}{4}
\put(40,40){\line(-1,0){20}} \put(40,30){\line(-1,0){20}}
\put(40,20){\line(-1,0){20}} \put(40,40){\line(0,-1){20}}
\put(30,40){\line(0,-1){20}} \put(20,40){\line(0,-1){20}}
\end{picture}
\quad
\begin{picture}(55,55)(0,10)
\gs{40}{40}{2}
\gs{40}{30}{2}
\bs{0}{40}{5}
\put(50,50){\line(-1,0){20}} \put(50,40){\line(-1,0){20}}
\put(50,30){\line(-1,0){20}} \put(50,50){\line(0,-1){20}}
\put(40,50){\line(0,-1){20}} \put(30,50){\line(0,-1){20}}
\end{picture}
\quad \raisebox{40pt}{
\parbox[t]{5cm}{
\begin{tabbing}

 $k_1 = k_2 = 1$ or $k_{r-1} = k_r = 1$  \quad  \= $c=1$ \\

\end{tabbing}
}}

\bigskip

\noindent \raisebox{31pt}{9a.}
\begin{picture}(45,45)
\gs{30}{30}{3}
\gs{30}{20}{2}
\bs{0}{30}{4}
\put(40,40){\line(-1,0){30}} \put(40,30){\line(-1,0){30}}
\put(40,20){\line(-1,0){20}} \put(40,40){\line(0,-1){20}}
\put(30,40){\line(0,-1){20}} \put(20,40){\line(0,-1){20}}
\put(10,40){\line(0,-1){10}}
\end{picture}
\quad \raisebox{30pt}{
\parbox[t]{10cm}{$c \geqslant 2$, because there are independent invariants of types ``square'' and ``triangle''}}

\noindent \raisebox{41pt}{9b.}
\begin{picture}(55,55)
\gs{40}{40}{3}
\gs{40}{30}{2}
\bs{0}{40}{5}
\put(50,50){\line(-1,0){30}} \put(50,40){\line(-1,0){30}}
\put(50,30){\line(-1,0){20}} \put(50,50){\line(0,-1){20}}
\put(40,50){\line(0,-1){20}} \put(30,50){\line(0,-1){20}}
\put(20,50){\line(0,-1){10}}
\end{picture}
\raisebox{50pt}{
\parbox[t]{5cm}{
\begin{tabbing}

 $k_1 = 1$ \quad \= $k_2=1$ \qquad \qquad \quad \= $k_3$ arbitrary \quad \= $k_4, k_5$ arbitrary \quad \=  $c=1$ \\
 $k_1 = 1$ \> $k_2$ arbitrary \> $k_3$ arbitrary \> $k_4=k_5=1$  \> $c=1$ \\

\end{tabbing}
}}


\noindent \raisebox{11pt}{10.}
\begin{picture}(60,25)
\gs{0}{10}{1}
\gs{45}{10}{3}
\gs{45}{0}{2}
\put(55,20){\line(0,-1){20}} \put(45,20){\line(0,-1){20}}
\put(35,20){\line(0,-1){20}} \put(25,20){\line(0,-1){10}}
\put(10,20){\line(0,-1){10}} \put(0,20){\line(0,-1){10}}
\put(55,20){\line(-1,0){30}} \put(55,10){\line(-1,0){30}}
\put(55,0){\line(-1,0){20}} \put(10,20){\line(-1,0){10}}
\put(10,10){\line(-1,0){10}}
\multiput(12,12)(5,0){3}{\circle*{1}}
\end{picture}

\noindent
This case appears only when the number of blocks in the first row is $s \leqslant 4$.

\noindent
  $s=3$: this is Case $9$.

\noindent
  $s=4$: $c \geqslant 2$, because there are independent invariants of types ``square'' and ``triangle''.
\smallskip

If there are at least $3$ blocks in the second row, then there are two independent invariants of type ``square'', i.e., $c \geqslant 2$.


\noindent \raisebox{31pt}{11a.}
\begin{picture}(45,45)
\gs{30}{30}{3}
\gs{30}{20}{1}
\gs{30}{10}{1}
\bs{0}{30}{4}
\put(40,40){\line(-1,0){30}} \put(40,30){\line(-1,0){30}}
\put(40,20){\line(-1,0){10}} \put(40,10){\line(-1,0){10}}
\put(40,40){\line(0,-1){30}} \put(30,40){\line(0,-1){30}}
\put(20,40){\line(0,-1){10}} \put(10,40){\line(0,-1){10}}
\end{picture}
\quad \raisebox{30pt}{
\parbox[t]{10cm}{$c \geqslant 2$, because there are two invariants of type ``triangle''}}

\noindent \raisebox{41pt}{11b.}
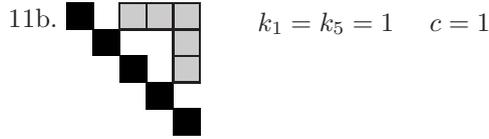
\begin{picture}(55,55)
\gs{40}{40}{3}
\gs{40}{30}{1}
\gs{40}{20}{1}
\bs{0}{40}{5}
\put(50,50){\line(-1,0){30}} \put(50,40){\line(-1,0){30}}
\put(50,30){\line(-1,0){10}} \put(50,20){\line(-1,0){10}}
\put(50,50){\line(0,-1){30}} \put(40,50){\line(0,-1){30}}
\put(30,50){\line(0,-1){10}} \put(20,50){\line(0,-1){10}}
\end{picture}
\quad \raisebox{50pt}{
\parbox[t]{5cm}{
\begin{tabbing}

 $k_1 = k_5 = 1$ \quad \= $c=1$ \\

\end{tabbing}
}}


\noindent \raisebox{21pt}{12.}
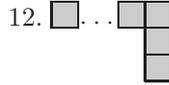
\begin{picture}(50,35)
\gs{0}{20}{1}
\gs{35}{20}{2}
\gs{35}{10}{1}
\gs{35}{0}{1}
\put(45,30){\line(0,-1){30}} \put(35,30){\line(0,-1){30}}
\put(25,30){\line(0,-1){10}} \put(10,30){\line(0,-1){10}}
\put(0,30){\line(0,-1){10}} \put(45,30){\line(-1,0){20}}
\put(45,20){\line(-1,0){20}} \put(45,10){\line(-1,0){10}}
\put(10,30){\line(-1,0){10}} \put(10,20){\line(-1,0){10}}
\put(45,0){\line(-1,0){10}}

\multiput(12,22)(5,0){3}{\circle*{1}}
\end{picture}

\noindent
This case appears only when the number of blocks in the first row is $s \leqslant 4$.

\noindent
 $s=3$: this is Case $11$.

\noindent
 $s=4$: $c \geqslant 2$, because there are two invariants of type ``triangle''.


\noindent \raisebox{41pt}{13.}
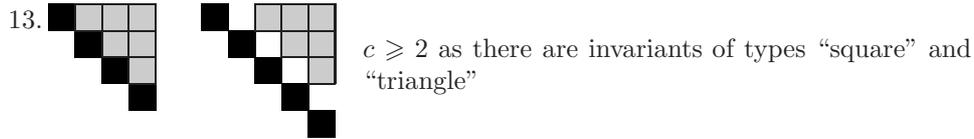
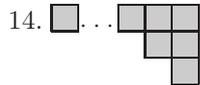
\begin{picture}(45,45)(0,-10)
\gs{30}{30}{3}
\gs{30}{20}{2}
\gs{30}{10}{1}
\bs{0}{30}{4}
\put(40,40){\line(-1,0){30}} \put(40,30){\line(-1,0){30}}
\put(40,20){\line(-1,0){20}} \put(40,10){\line(-1,0){10}}
\put(40,40){\line(0,-1){30}} \put(30,40){\line(0,-1){30}}
\put(20,40){\line(0,-1){20}} \put(10,40){\line(0,-1){10}}
\end{picture}
\quad
\begin{picture}(55,55)
\gs{40}{40}{3}
\gs{40}{30}{2}
\gs{40}{20}{1}
\bs{0}{40}{5}
\put(50,50){\line(-1,0){30}} \put(50,40){\line(-1,0){30}}
\put(50,30){\line(-1,0){20}} \put(50,20){\line(-1,0){10}}
\put(50,50){\line(0,-1){30}} \put(40,50){\line(0,-1){30}}
\put(30,50){\line(0,-1){20}} \put(20,50){\line(0,-1){10}}
\end{picture}
\raisebox{30pt}{
\parbox[t]{8cm}{$c \geqslant 2$ as there are invariants of types ``square'' and ``triangle''}}


\noindent \raisebox{21pt}{14.}
\begin{picture}(60,35)
\gs{0}{20}{1}
\gs{45}{20}{3}
\gs{45}{10}{2}
\gs{45}{0}{1}
\put(55,30){\line(0,-1){30}} \put(45,30){\line(0,-1){30}}
\put(35,30){\line(0,-1){20}} \put(25,30){\line(0,-1){10}}
\put(10,30){\line(0,-1){10}} \put(0,30){\line(0,-1){10}}
\put(55,30){\line(-1,0){30}} \put(55,20){\line(-1,0){30}}
\put(55,10){\line(-1,0){20}} \put(10,30){\line(-1,0){10}}
\put(10,20){\line(-1,0){10}} \put(55,0){\line(-1,0){10}}
\multiput(12,22)(5,0){3}{\circle*{1}}
\end{picture}

\noindent
This case appears only when the number of blocks in the first row is $s \leqslant 4$.

\noindent
 $s=3$: this is Case $13$.

\noindent
 $s=4$: $c \geqslant 2$, since this case can be reduced to Case $13$ by Lemma~\ref{subblocks}


\noindent \raisebox{26pt}{15.}
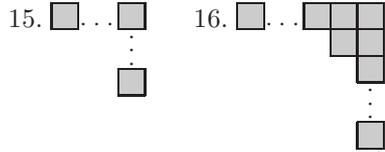
\begin{picture}(40,40)(0,0)
\gs{0}{25}{1}
\gs{25}{25}{1}
\gs{25}{0}{1}
\put(35,35){\line(0,-1){10}} \put(25,35){\line(0,-1){10}}
\put(10,35){\line(0,-1){10}} \put(0,35){\line(0,-1){10}}
\put(25,10){\line(0,-1){10}} \put(35,10){\line(0,-1){10}}
\put(35,35){\line(-1,0){10}} \put(35,25){\line(-1,0){10}}
\put(10,35){\line(-1,0){10}} \put(10,25){\line(-1,0){10}}
\put(35,10){\line(-1,0){10}} \put(35,0){\line(-1,0){10}}
\multiput(12,27)(5,0){3}{\circle*{1}}
\multiput(30,12)(0,5){3}{\circle*{1}}
\end{picture}
\quad
\raisebox{26pt}{16.}
\begin{picture}(60,60)(0,20)
\gs{0}{45}{1}
\gs{45}{45}{3}
\gs{45}{35}{2}
\gs{45}{25}{1}
\gs{45}{0}{1}
\put(55,55){\line(0,-1){30}} \put(45,55){\line(0,-1){30}}
\put(35,55){\line(0,-1){20}} \put(25,55){\line(0,-1){10}}
\put(10,55){\line(0,-1){10}} \put(0,55){\line(0,-1){10}}
\put(55,10){\line(0,-1){10}} \put(45,10){\line(0,-1){10}}
\put(55,55){\line(-1,0){30}} \put(55,45){\line(-1,0){30}}
\put(55,35){\line(-1,0){20}} \put(55,25){\line(-1,0){10}}
\put(10,55){\line(-1,0){10}} \put(10,45){\line(-1,0){10}}
\put(55,10){\line(-1,0){10}} \put(55,0){\line(-1,0){10}}
\multiput(12,47)(5,0){3}{\circle*{1}}
\multiput(50,12)(0,5){3}{\circle*{1}}
\end{picture}

\bigskip

\medskip

\noindent
These cases appear only when $|s-m| \leqslant~1$ (here $s$, $m$ are the numbers of blocks in the first row and in the last column, respectively). If $s,m \geqslant 4$ (these are the remaining cases), then the complexity is at least $2$, because there are two invariants of type ``triangle'' (Case $15$) or invariants of types ``triangle'' and ``square'' (Case~$16$).

Combining all cases together we obtain the classification given in Table \ref{sl}. Recall that the classification is given up to transposition with respect to the secondary diagonal and permutation of parabolics.

\subsection{Group $\mathbf{SO_n}$}

Consider the cases, where $P$ and $Q$ are not special.


\noindent \raisebox{11pt}{1.}
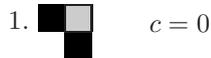
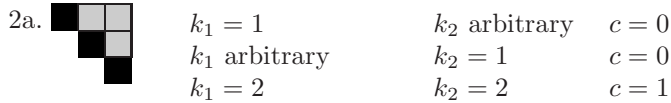
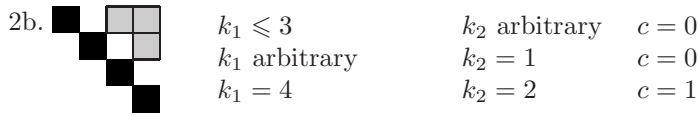
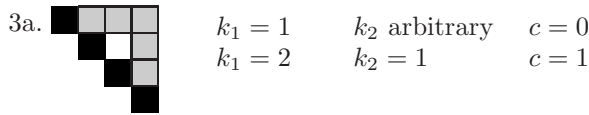
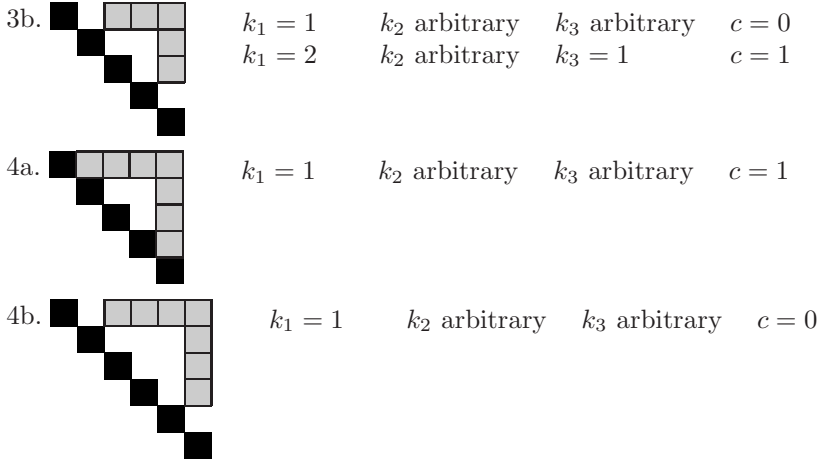
\begin{picture}(25,25)
\gs{10}{10}{1}
\bs{0}{10}{2}
\put(10,10){\line(1,0){10}} \put(10,20){\line(1,0){10}}
\put(10,10){\line(0,1){10}} \put(20,10){\line(0,1){10}}
\end{picture}
\quad \raisebox{10pt}{ $c=0$ }


\noindent \raisebox{21pt}{2a.}
\begin{picture}(35,35)
\gs{20}{20}{2}
\gs{20}{10}{1}
\bs{0}{20}{3}
\put(30,30){\line(-1,0){20}} \put(30,20){\line(-1,0){20}}
\put(30,10){\line(-1,0){10}} \put(30,30){\line(0,-1){20}}
\put(20,30){\line(0,-1){20}} \put(10,30){\line(0,-1){10}}
\end{picture}
\quad \raisebox{30pt}{
\parbox[t]{5cm}{

\begin{tabbing}
$k_1=1$ \qquad \qquad \qquad  \= $k_2$ arbitrary \quad   \=  $c=0$ \\
$k_1$ arbitrary
     \> $k_2=1$                         \> $c=0$   \\
$k_1=2$              \> $k_2=2$                          \> $c=1$   \\

\end{tabbing}
}}

\noindent \raisebox{31pt}{2b.}
\begin{picture}(45,45)
\gs{30}{30}{2}
\gs{30}{20}{1}
\bs{0}{30}{4}
\put(40,40){\line(-1,0){20}} \put(40,30){\line(-1,0){20}}
\put(40,20){\line(-1,0){10}} \put(40,40){\line(0,-1){20}}
\put(30,40){\line(0,-1){20}} \put(20,40){\line(0,-1){10}}
\end{picture}
\quad \raisebox{40pt}{
\parbox[t]{5cm}{

\begin{tabbing}
$k_1 \leqslant 3$ \qquad \qquad \qquad   \=  $k_2$ arbitrary \quad \= $c=0$   \\
$k_1$ arbitrary            \>  $k_2=1$                 \> $c=0$   \\
$k_1=4$                      \>  $k_2=2$                 \> $c=1$   \\
\end{tabbing}

}}


\noindent \raisebox{31pt}{3a.}
\begin{picture}(45,45)
\gs{30}{30}{3}
\gs{30}{20}{1}
\gs{30}{10}{1}
\bs{0}{30}{4}
\put(40,40){\line(-1,0){30}} \put(40,30){\line(-1,0){30}}
\put(40,20){\line(-1,0){10}} \put(40,10){\line(-1,0){10}}
\put(40,40){\line(0,-1){30}} \put(30,40){\line(0,-1){30}}
\put(20,40){\line(0,-1){10}} \put(10,40){\line(0,-1){10}}
\end{picture}
\quad \raisebox{40pt}{
\parbox[t]{5cm}{
\begin{tabbing}
$k_1=1$ \qquad      \=   $k_2$ arbitrary \quad \=   $c=0$  \\

$k_1=2$             \> $k_2=1$            \>   $c=1$  \\

\end{tabbing}

}}

\noindent \raisebox{41pt}{3b.}
\begin{picture}(55,55)
\gs{40}{40}{3}
\gs{40}{30}{1}
\gs{40}{20}{1}
\bs{0}{40}{5}
\put(50,50){\line(-1,0){30}} \put(50,40){\line(-1,0){30}}
\put(50,30){\line(-1,0){10}} \put(50,20){\line(-1,0){10}}
\put(50,50){\line(0,-1){30}} \put(40,50){\line(0,-1){30}}
\put(30,50){\line(0,-1){10}} \put(20,50){\line(0,-1){10}}
\end{picture}
\quad \raisebox{50pt}{
\parbox[t]{5cm}{
\begin{tabbing}

$k_1=1$ \qquad   \= $k_2$ arbitrary \quad   \= $k_3$ arbitrary \quad    \=  $c=0$  \\

$k_1=2$          \> $k_2$ arbitrary         \> $k_3=1$                    \>  $c=1$  \\

\end{tabbing}

}}


\noindent \raisebox{41pt}{4a.}
\begin{picture}(55,55)

\gs{40}{40}{4}
\gsdown{40}{40}{4}
\bs{0}{40}{5}

\put(50,50){\line(-1,0){40}} \put(50,40){\line(-1,0){40}}
\put(50,30){\line(-1,0){10}} \put(50,20){\line(-1,0){10}}
\put(50,10){\line(-1,0){10}} \put(50,50){\line(0,-1){40}}
\put(40,50){\line(0,-1){40}} \put(30,50){\line(0,-1){10}}
\put(20,50){\line(0,-1){10}} \put(10,50){\line(0,-1){10}}

\end{picture}
\quad \raisebox{50pt}{
\parbox[t]{5cm}{
\begin{tabbing}
$k_1=1$ \qquad \= $k_2$ arbitrary \quad \=  $k_3$ arbitrary \quad \=   $c=1$   \\

\end{tabbing}

}}

\noindent \raisebox{51pt}{4b.}
\begin{picture}(65,65)
\gs{50}{50}{4}
\gsdown{50}{50}{4}
\bs{0}{50}{6}

\put(60,60){\line(-1,0){40}} \put(60,50){\line(-1,0){40}}
\put(60,40){\line(-1,0){10}} \put(60,30){\line(-1,0){10}}
\put(60,20){\line(-1,0){10}} \put(60,60){\line(0,-1){40}}
\put(50,60){\line(0,-1){40}} \put(40,60){\line(0,-1){10}}
\put(30,60){\line(0,-1){10}} \put(20,60){\line(0,-1){10}}

\end{picture}
\quad \raisebox{60pt}{
\parbox[t]{5cm}{
\begin{tabbing}
$k_1=1$ \qquad \= $k_2$ arbitrary \quad  \= $k_3$ arbitrary \quad  \=  $c=0$            \\

\end{tabbing}

}}

If there are at least $5$ blocks in the first row, then by Lemma~\ref{in_one_row_3_4} their height cannot be greater than one.

Suppose nonzero blocks stand only in the first row and in the last column (denote the number of blocks in the first row by $m$) and suppose that their height is $k_1 = 1$. Denote the complexity for this case by $c_{m,a}$ if $r = m + 1$ and by $c_{m,b}$ if $r = m + 2$ (here $r$ is the number of diagonal blocks). Then we have the following lemma.

\begin{lemma}
Suppose $m \geqslant 4$. We have $c_{m,a}=c_{m-1,b}+1$, $c_{m,b}=c_{m-1,a}$.
\label{so_k1=1}
\end{lemma}

\begin{proof}
We can put a general  pair of blocks $X_{1,i}$ and $X_{1,r+1-i}$ ($i \neq \frac{r+1}{2}$) to the form  $(x, 0, \dots, 0)$ and $(t, 0, \dots, 0, y)$  by multiplication on the right if the widths of the blocks are at least $2$, and to the form  $(x)$ and $(y)$ if the widths are equal to $1$. Here $t$, $x$ are any nonzero numbers and $y$ is determined by $x$, namely the product $xy$ is invariant. A nonzero general block $X_{1, \frac{r+1}{2}}$ (if $r$ is odd) can be put in the form  $(x, 0 \dots 0, y)$ if the width of the block is at least $2$, (where  $x$ and $y$ are as above). If the width of this block equals $1$ then we cannot change the entry in this block. All these matrices are multiplied on the left by one and the same number. Thus the complexity equals the number of these pairs plus the number of central blocks ($0$ or~$1$) minus~$1$. From this, we obtain the lemma.
\end{proof}

In Cases $5$ and $6$ we can easily compute complexities using this lemma, and $c \geqslant 2$ for $m \geqslant 7$.


\noindent \raisebox{51pt}{5a.}
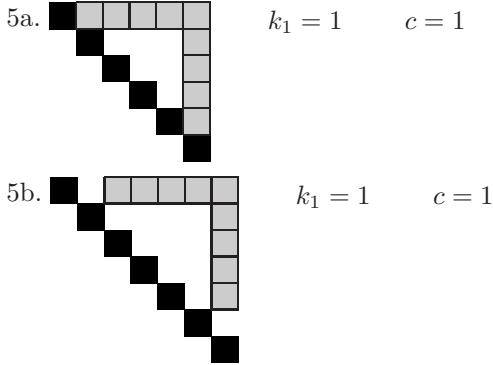
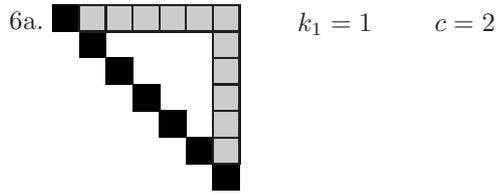
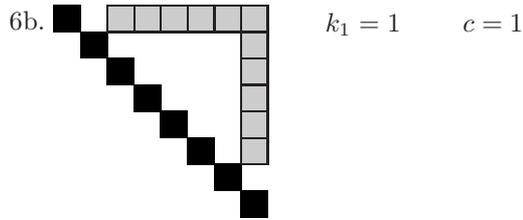
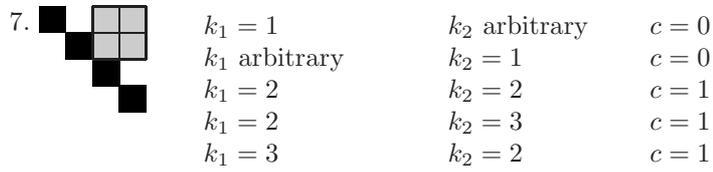
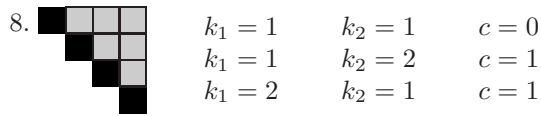
\begin{picture}(65,65)
\gs{50}{50}{5}
\gsdown{50}{50}{5}
\bs{0}{50}{6}

\put(60,60){\line(-1,0){50}} \put(60,50){\line(-1,0){50}}
\put(60,40){\line(-1,0){10}} \put(60,30){\line(-1,0){10}}
\put(60,20){\line(-1,0){10}} \put(60,10){\line(-1,0){10}}
\put(60,60){\line(0,-1){50}} \put(50,60){\line(0,-1){50}}
\put(40,60){\line(0,-1){10}} \put(30,60){\line(0,-1){10}}
\put(20,60){\line(0,-1){10}} \put(10,60){\line(0,-1){10}}

\end{picture}
\quad \raisebox{50pt}{
\parbox[t]{5cm}{
$k_1=1$ \qquad  $c=1$

}}

\noindent \raisebox{61pt}{5b.}
\begin{picture}(75,75)
\gs{60}{60}{5}
\gsdown{60}{60}{5}
\bs{0}{60}{7}

\put(70,70){\line(-1,0){50}} \put(70,60){\line(-1,0){50}}
\put(70,50){\line(-1,0){10}} \put(70,40){\line(-1,0){10}}
\put(70,30){\line(-1,0){10}} \put(70,20){\line(-1,0){10}}
\put(70,70){\line(0,-1){50}} \put(60,70){\line(0,-1){50}}
\put(50,70){\line(0,-1){10}} \put(40,70){\line(0,-1){10}}
\put(30,70){\line(0,-1){10}} \put(20,70){\line(0,-1){10}}

\end{picture}
\quad \raisebox{60pt}{
\parbox[t]{5cm}{
$k_1=1$ \qquad  $c=1$

}}


\noindent \raisebox{61pt}{6a.}
\begin{picture}(75,75)
\gs{60}{60}{6}
\gsdown{60}{60}{6}
\bs{0}{60}{7}

\put(70,70){\line(-1,0){60}} \put(70,60){\line(-1,0){60}}
\put(70,50){\line(-1,0){10}} \put(70,40){\line(-1,0){10}}
\put(70,30){\line(-1,0){10}} \put(70,20){\line(-1,0){10}}
\put(70,10){\line(-1,0){10}} \put(70,70){\line(0,-1){60}}
\put(60,70){\line(0,-1){60}} \put(50,70){\line(0,-1){10}}
\put(40,70){\line(0,-1){10}} \put(30,70){\line(0,-1){10}}
\put(20,70){\line(0,-1){10}} \put(10,70){\line(0,-1){10}}

\end{picture}
\quad \raisebox{60pt}{
\parbox[t]{5cm}{
$k_1=1$ \qquad $c = 2$ }}

\noindent \raisebox{71pt}{6b.}
\begin{picture}(85,85)
\gs{70}{70}{6}
\gsdown{70}{70}{6}
\bs{0}{70}{8}

\put(80,80){\line(-1,0){60}} \put(80,70){\line(-1,0){60}}
\put(80,60){\line(-1,0){10}} \put(80,50){\line(-1,0){10}}
\put(80,40){\line(-1,0){10}} \put(80,30){\line(-1,0){10}}
\put(80,20){\line(-1,0){10}} \put(80,80){\line(0,-1){60}}
\put(70,80){\line(0,-1){60}} \put(60,80){\line(0,-1){10}}
\put(50,80){\line(0,-1){10}} \put(40,80){\line(0,-1){10}}
\put(30,80){\line(0,-1){10}} \put(20,80){\line(0,-1){10}}

\end{picture}
\quad \raisebox{70pt}{
\parbox[t]{5cm}{
$k_1=1$ \qquad  $c=1$ }}


\noindent \raisebox{31pt}{7.}
\begin{picture}(45,45)
\gs{30}{30}{2}
\gs{30}{20}{2}
\bs{0}{30}{4}
\put(40,40){\line(-1,0){20}} \put(40,30){\line(-1,0){20}}
\put(40,20){\line(-1,0){20}} \put(40,40){\line(0,-1){20}}
\put(30,40){\line(0,-1){20}} \put(20,40){\line(0,-1){20}}
\end{picture}
\quad \raisebox{40pt}{
\parbox[t]{10cm}{
\begin{tabbing}
$k_1=1$ \qquad  \qquad \qquad  \=  $k_2$ arbitrary \qquad  \=  $c=0$ \\
$k_1$ arbitrary \> $k_2=1$                \>  $c=0$ \\
$k_1=2$           \>  $k_2=2$                   \> $c=1$  \\
$k_1=2$           \> $k_2=3$                \> $c=1$  \\
$k_1=3$           \> $k_2=2$                     \> $c = 1$ \\

\end{tabbing}

}}


\noindent \raisebox{31pt}{8.}
\begin{picture}(45,45)

\gs{30}{30}{3}
\gs{30}{20}{2}
\gs{30}{10}{1}
\bs{0}{30}{4}
\put(40,40){\line(-1,0){30}} \put(40,30){\line(-1,0){30}}
\put(40,20){\line(-1,0){20}} \put(40,10){\line(-1,0){10}}
\put(40,40){\line(0,-1){30}} \put(30,40){\line(0,-1){30}}
\put(20,40){\line(0,-1){20}} \put(10,40){\line(0,-1){10}}
\end{picture}
\quad \raisebox{40pt}{
\parbox[t]{5cm}{
\begin{tabbing}

$k_1=1$ \qquad  \=  $k_2=1$ \qquad   \=   $c=0$  \\

$k_1=1$     \>  $k_2=2$          \>  $c=1$   \\

$k_1=2$         \>  $k_2=1$         \>    $c=1$   \\

\end{tabbing}

}}

If we add blocks only to the first row (and add blocks symmetrical to them with respect to the secondary diagonal), then these cases do not appear.


\noindent \raisebox{41pt}{9a.}
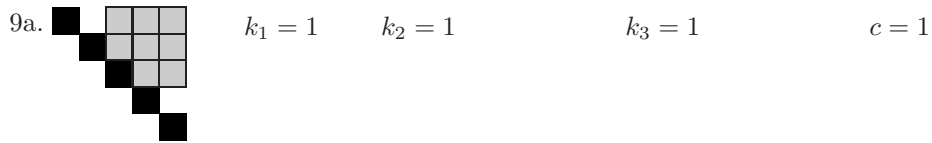
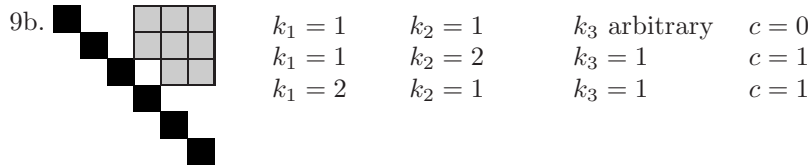
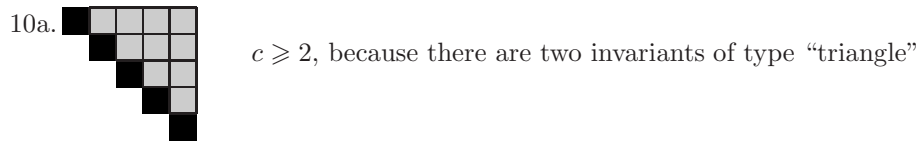
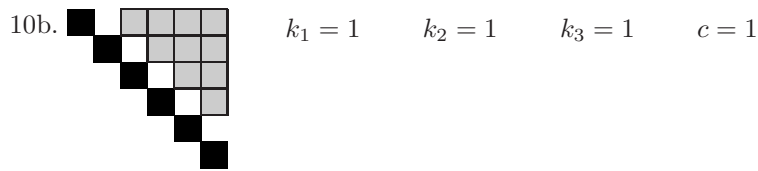
\begin{picture}(55,55)
\gs{40}{40}{3}
\gs{40}{30}{3}
\gs{40}{20}{2}
\bs{0}{40}{5}

\put(50,50){\line(-1,0){30}} \put(50,40){\line(-1,0){30}}
\put(50,30){\line(-1,0){30}} \put(50,20){\line(-1,0){20}}
\put(50,50){\line(0,-1){30}} \put(40,50){\line(0,-1){30}}
\put(30,50){\line(0,-1){30}} \put(20,50){\line(0,-1){20}}

\end{picture}
\quad \raisebox{50pt}{
\parbox[t]{5cm}{
%
%
\begin{tabbing}

$k_1=1$ \qquad       \= $k_2=1$ \qquad \qquad \qquad \= $k_3=1$ \qquad \qquad \qquad \=   $c=1$  \\

\end{tabbing}

}}

\noindent \raisebox{51pt}{9b.}
\begin{picture}(65,65)
\gs{50}{50}{3}
\gs{50}{40}{3}
\gs{50}{30}{2}
\bs{0}{50}{6}

\put(60,60){\line(-1,0){30}} \put(60,50){\line(-1,0){30}}
\put(60,40){\line(-1,0){30}} \put(60,30){\line(-1,0){20}}
\put(60,60){\line(0,-1){30}} \put(50,60){\line(0,-1){30}}
\put(40,60){\line(0,-1){30}} \put(30,60){\line(0,-1){20}}

\end{picture}
\quad \raisebox{60pt}{
\parbox[t]{5cm}{
%
%
\begin{tabbing}

$k_1=1$ \qquad   \= $k_2=1$ \qquad \quad \= $k_3$ arbitrary \quad  \=  $c=0$   \\

$k_1=1$          \> $k_2=2$      \> $k_3=1$                \> $c=1$   \\

$k_1=2$          \> $k_2=1$        \> $k_3=1$                \>  $c=1$  \\

\end{tabbing}

}}


\noindent \raisebox{41pt}{10a.}
\begin{picture}(55,55)
\gs{40}{40}{4}
\gs{40}{30}{3}
\gs{40}{20}{2}
\gs{40}{10}{1}
\bs{0}{40}{5}

\put(50,50){\line(-1,0){40}} \put(50,40){\line(-1,0){40}}
\put(50,30){\line(-1,0){30}} \put(50,20){\line(-1,0){20}}
\put(50,10){\line(-1,0){10}} \put(50,50){\line(0,-1){40}}
\put(40,50){\line(0,-1){40}} \put(30,50){\line(0,-1){30}}
\put(20,50){\line(0,-1){20}} \put(10,50){\line(0,-1){10}}

\end{picture}
\quad \raisebox{30pt}{
\parbox[t]{12cm}{
$c \geqslant 2$, because there are two invariants of type ``triangle''
}}

\noindent \raisebox{51pt}{10b.}
\begin{picture}(65,65)
\gs{50}{50}{4}
\gs{50}{40}{3}
\gs{50}{30}{2}
\gs{50}{20}{1}
\bs{0}{50}{6}

\put(60,60){\line(-1,0){40}} \put(60,50){\line(-1,0){40}}
\put(60,40){\line(-1,0){30}} \put(60,30){\line(-1,0){20}}
\put(60,20){\line(-1,0){10}} \put(60,60){\line(0,-1){40}}
\put(50,60){\line(0,-1){40}} \put(40,60){\line(0,-1){30}}
\put(30,60){\line(0,-1){20}} \put(20,60){\line(0,-1){10}}

\end{picture}
\quad \raisebox{60pt}{
\parbox[t]{5cm}{
\begin{tabbing}

$k_1=1$ \qquad   \= $k_2=1$ \qquad \= $k_3=1$ \qquad  \=  $c=1$   \\
\end{tabbing}

}}


\noindent \raisebox{51pt}{11.}
\begin{picture}(65,65)
\gs{50}{50}{5}
\gs{50}{40}{3}
\gs{50}{30}{2}
\gs{50}{20}{1}
\gs{50}{10}{1}
\bs{0}{50}{6}

\put(60,60){\line(-1,0){50}} \put(60,50){\line(-1,0){50}}
\put(60,40){\line(-1,0){30}} \put(60,30){\line(-1,0){20}}
\put(60,20){\line(-1,0){10}} \put(60,10){\line(-1,0){10}}
\put(60,60){\line(0,-1){50}} \put(50,60){\line(0,-1){50}}
\put(40,60){\line(0,-1){30}} \put(30,60){\line(0,-1){20}}
\put(20,60){\line(0,-1){10}} \put(10,60){\line(0,-1){10}}

\end{picture}
\quad \raisebox{50pt}{
\parbox[t]{12cm}{

$c \geqslant 2$, because there are two invariants of type ``triangle''

}}


\noindent If we add blocks only to the first row (and add blocks symmetrical to them), then these cases do not appear.


\noindent \raisebox{51pt}{12a.}
\begin{picture}(65,65)

\gs{50}{50}{4}
\gs{50}{40}{4}
\gs{50}{30}{2}
\gs{50}{20}{2}
\bs{0}{50}{6}

\put(60,60){\line(-1,0){40}} \put(60,50){\line(-1,0){40}}
\put(60,40){\line(-1,0){40}} \put(60,30){\line(-1,0){20}}
\put(60,20){\line(-1,0){20}} \put(60,60){\line(0,-1){40}}
\put(50,60){\line(0,-1){40}} \put(40,60){\line(0,-1){40}}
\put(30,60){\line(0,-1){20}} \put(20,60){\line(0,-1){20}}

\end{picture}
\quad \raisebox{50pt}{
\parbox[t]{9cm}{
$c \geqslant 2$, because there are invariants of types ``square'' and ``triangle''
}}

\noindent \raisebox{61pt}{12b.}
\begin{picture}(75,75)
\gs{60}{60}{4}
\gs{60}{50}{4}
\gs{60}{40}{2}
\gs{60}{30}{2}
\bs{0}{60}{7}

\put(70,70){\line(-1,0){40}} \put(70,60){\line(-1,0){40}}
\put(70,50){\line(-1,0){40}} \put(70,40){\line(-1,0){20}}
\put(70,30){\line(-1,0){20}} \put(70,70){\line(0,-1){40}}
\put(60,70){\line(0,-1){40}} \put(50,70){\line(0,-1){40}}
\put(40,70){\line(0,-1){20}} \put(30,70){\line(0,-1){20}}

\end{picture}
\quad \raisebox{60pt}{
\parbox[t]{9cm}{
$c \geqslant 2$, because there are invariants of types ``square'' and ``triangle''
}}

\noindent If we add one or two blocks to the first row (and add blocks symmetrical to them), then $c \geqslant 2$ (reduces to Case $12$).

\noindent If we add more than two blocks to the first row (and add blocks symmetrical to them), then these cases do not appear.

\noindent If we add blocks to the second row (and to the first row, respectively, and add blocks symmetrical to them), then there are two invariants of type ``square''.


\noindent \raisebox{51pt}{13.}
\begin{picture}(65,65)
\gs{50}{50}{3}
\gs{50}{40}{3}
\gs{50}{30}{3}
\bs{0}{50}{6}

\put(60,60){\line(-1,0){30}} \put(60,50){\line(-1,0){30}}
\put(60,40){\line(-1,0){30}} \put(60,30){\line(-1,0){30}}
\put(60,60){\line(0,-1){30}} \put(50,60){\line(0,-1){30}}
\put(40,60){\line(0,-1){30}} \put(30,60){\line(0,-1){30}}

\end{picture}
\quad \raisebox{60pt}{
\parbox[t]{10cm}{

\begin{tabbing}

$k_1=1$ \qquad \= $k_2=1$ \qquad \= $k_3=1$ \qquad   \= $c=0$ \\

$k_1=1$        \> $k_2=1$        \> $k_3=2$           \> $c=1$ \\

$k_1=1$        \> $k_2=2$        \> $k_3=1$           \> $c=1$ \\

$k_1=2$        \> $k_2=1$        \> $k_3=1$           \> $c=1$ \\

\end{tabbing}
}}


\noindent \raisebox{51pt}{14.}
\begin{picture}(65,65)
\gs{50}{50}{4}
\gs{50}{40}{3}
\gs{50}{30}{3}
\gs{50}{20}{1}
\bs{0}{50}{6}

\put(60,60){\line(-1,0){40}} \put(60,50){\line(-1,0){40}}
\put(60,40){\line(-1,0){30}} \put(60,30){\line(-1,0){30}}
\put(60,20){\line(-1,0){10}} \put(60,60){\line(0,-1){40}}
\put(50,60){\line(0,-1){40}} \put(40,60){\line(0,-1){30}}
\put(30,60){\line(0,-1){30}} \put(20,60){\line(0,-1){10}}

\end{picture}
\quad \raisebox{60pt}{
\parbox[t]{5cm}{
%
%
%
\begin{tabbing}

$k_1=1$ \qquad \= $k_2=1$ \qquad \= $k_3=1$ \qquad   \= $c=1$ \\

\end{tabbing}
}}


\noindent \raisebox{51pt}{15.}
\begin{picture}(65,65)
\gs{50}{50}{5}
\gs{50}{40}{3}
\gs{50}{30}{3}
\gs{50}{20}{1}
\gs{50}{10}{1}
\bs{0}{50}{6}

\put(60,60){\line(-1,0){50}} \put(60,50){\line(-1,0){50}}
\put(60,40){\line(-1,0){30}} \put(60,30){\line(-1,0){30}}
\put(60,20){\line(-1,0){10}} \put(60,10){\line(-1,0){10}}
\put(60,60){\line(0,-1){50}} \put(50,60){\line(0,-1){50}}
\put(40,60){\line(0,-1){30}} \put(30,60){\line(0,-1){30}}
\put(20,60){\line(0,-1){10}} \put(10,60){\line(0,-1){10}}

\end{picture}
\quad \raisebox{50pt}{
\parbox[t]{12cm}{
$c \geqslant 2$, because this case reduces to Case $11$ by Lemma~\ref{subblocks}
}}


\noindent If we add  blocks to the first row (and add blocks symmetrical to them), then these cases do not appear.

\noindent If there are $3$ blocks in the third row, $4$ blocks in the second row, and $4$ or $5$ in the first row, then $c \geqslant 2$ (this follows from Case $12$a). If there are $\geqslant 6$ blocks in the first row, then these cases do not appear.

\noindent If there are $\geqslant 5$ blocks in the second row, then there are two invariants of type ``square''.

\noindent If there are  $\geqslant 4$ blocks in the third row then there are two invariants of type ``square''.

\bigskip
Now consider the case of special subgroups.

The numeration corresponds to the numeration of cases for non-special subgroups, obtained by replacing a special subgroup with a non-special one by conjugation with transposition of two middle basis vectors and by replacing two middle blocks with one central block. It is sufficient to consider the cases for which the size of two middle blocks obtained by this transformation is not greater than the size of two middle blocks for another subgroup.

\medskip

\noindent \raisebox{11pt}{0.}
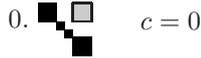
\begin{picture}(25,25)

\put(13,13){\graybox{7}{7}}

\put(20,20){\line(0,-1){7}} \put(20,20) {\line(-1,0){7}}
\put(13,13){\line(0,1){7}} \put(13,13) {\line(1,0){7}}

\put(0,13){\blackbox{7}{7}}
\put(7,10){\blackbox{3}{3}}
\put(10,7){\blackbox{3}{3}}
\put(13,0){\blackbox{7}{7}}

\end{picture}\quad \raisebox{10pt}{ $c=0$ }

\medskip

\noindent \raisebox{31pt}{2b.}
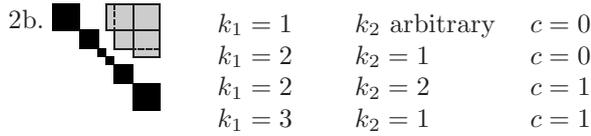
\begin{picture}(45,45)
\put(30,30){\graybox{10}{10}}
\put(20,30){\graybox{10}{10}}
\put(30,20){\graybox{10}{10}}
\put(23,23){\graybox{7}{7}}

\put(40,40){\line(-1,0){20}} \put(40,30){\line(-1,0){20}}
\put(40,20){\line(-1,0){10}} \put(40,40){\line(0,-1){20}}
\put(30,40){\line(0,-1){20}} \put(20,40){\line(0,-1){10}}

\put(23,23){\line(0,1){7}} \put(23,23){\line(1,0){7}}

\put(23,30) {\line(0,1){1}} \put(23,32) {\line(0,1){2}} \put(23,35)
{\line(0,1){2}} \put(23,38) {\line(0,1){2}}

\put(30,23) {\line(1,0){1}} \put(32,23) {\line(1,0){2}} \put(35,23)
{\line(1,0){2}} \put(38,23) {\line(1,0){2}}

\put(0,30){\blackbox{10}{10}}
\put(10,23){\blackbox{7}{7}}
\put(17,20){\blackbox{3}{3}}
\put(20,17){\blackbox{3}{3}}
\put(23,10){\blackbox{7}{7}}
\put(30,0){\blackbox{10}{10}}

\end{picture}
\quad \raisebox{40pt}{
\parbox[t]{5cm}{

\begin{tabbing}
$k_1=1$ \qquad \= $k_2$ arbitrary \quad \=  $c=0$ \\

$k_1=2$   \> $k_2=1$    \> $c=0$   \\

$k_1=2$   \> $k_2=2$    \> $c=1$   \\

$k_1=3$   \> $k_2=1$    \> $c=1$   \\

\end{tabbing}

}}


\noindent \raisebox{31pt}{3a.}
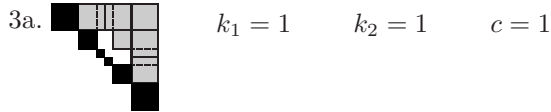
\begin{picture}(45,45)

\put(30,30){\graybox{10}{10}}
\put(20,30){\graybox{10}{10}}
\put(10,30){\graybox{10}{10}}
\put(30,20){\graybox{10}{10}}
\put(30,10){\graybox{10}{10}}
\put(23,23){\graybox{7}{7}}

\put(40,40){\line(-1,0){30}} \put(40,30){\line(-1,0){30}}
\put(40,20){\line(-1,0){10}} \put(40,10){\line(-1,0){10}}
\put(40,40){\line(0,-1){30}} \put(30,40){\line(0,-1){30}}
\put(20,40){\line(0,-1){10}} \put(10,40){\line(0,-1){10}}

\put(23,23){\line(0,1){7}} \put(23,23){\line(1,0){7}}

\put(23,30) {\line(0,1){1}} \put(23,32) {\line(0,1){2}} \put(23,35)
{\line(0,1){2}} \put(23,38) {\line(0,1){2}}

\put(17,30) {\line(0,1){1}} \put(17,32) {\line(0,1){2}} \put(17,35)
{\line(0,1){2}} \put(17,38) {\line(0,1){2}}

\put(30,23) {\line(1,0){1}} \put(32,23) {\line(1,0){2}} \put(35,23)
{\line(1,0){2}} \put(38,23) {\line(1,0){2}}

\put(30,17) {\line(1,0){1}} \put(32,17) {\line(1,0){2}} \put(35,17)
{\line(1,0){2}} \put(38,17) {\line(1,0){2}}

\put(0,30){\blackbox{10}{10}}
\put(10,23){\blackbox{7}{7}}
\put(17,20){\blackbox{3}{3}}
\put(20,17){\blackbox{3}{3}}
\put(23,10){\blackbox{7}{7}}
\put(30,0){\blackbox{10}{10}}

\end{picture}
\quad \raisebox{40pt}{
\parbox[t]{5cm}{

\begin{tabbing}
$k_1=1$ \qquad   \= $k_2=1$ \qquad   \= $c=1$  \\
\end{tabbing}
}}


\noindent \raisebox{51pt}{9b.}
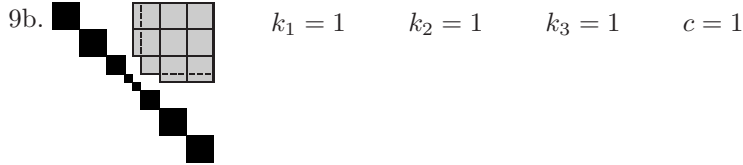
\begin{picture}(65,65)

\put(30,40){\graybox{30}{20}}
\put(40,30){\graybox{20}{10}}
\put(33,33){\graybox{7}{7}}

\put(60,60){\line(-1,0){30}} \put(60,50){\line(-1,0){30}}
\put(60,40){\line(-1,0){30}} \put(60,30){\line(-1,0){20}}
\put(60,60){\line(0,-1){30}} \put(50,60){\line(0,-1){30}}
\put(40,60){\line(0,-1){30}} \put(30,60){\line(0,-1){20}}

\put(33,33){\line(0,1){7}} \put(33,33){\line(1,0){7}}

\put(33,41) {\line(0,1){2}} \put(33,44) {\line(0,1){2}} \put(33,47)
{\line(0,1){2}} \put(33,51) {\line(0,1){2}} \put(33,54)
{\line(0,1){2}} \put(33,57) {\line(0,1){2}}

\put(41,33) {\line(1,0){2}} \put(44,33) {\line(1,0){2}} \put(47,33)
{\line(1,0){2}} \put(51,33) {\line(1,0){2}} \put(54,33)
{\line(1,0){2}} \put(57,33) {\line(1,0){2}}

\put(0,50){\blackbox{10}{10}}
\put(10,40){\blackbox{10}{10}}
\put(20,33){\blackbox{7}{7}}
\put(27,30){\blackbox{3}{3}}
\put(30,27){\blackbox{3}{3}}
\put(33,20){\blackbox{7}{7}}
\put(40,10){\blackbox{10}{10}}
\put(50,0){\blackbox{10}{10}}

\end{picture}
\quad \raisebox{60pt}{
\parbox[t]{5cm}{

\begin{tabbing}
$k_1=1$ \qquad   \= $k_2=1$ \qquad   \= $k_3=1$ \qquad   \= $c=1$  \\

\end{tabbing}

}}

\noindent {10b.} If we want the complexity to be $\leqslant 1$, then it is necessary to have  $k_3 = 1$. As one middle block of size $2$ is the same as two middle blocks of sizes $1$ and $1$, then we may assume that in Case $10$b (for non-special subgroups) both subgroups have two middle blocks, so we do not need to consider this case.

Combining all cases together we obtain the classification given in Table~\ref{so}.

\subsection{Group $\mathbf{Sp_n}$}

The numbers of cases correspond to the numbers of cases for $SO_n$.


\noindent \raisebox{11pt}{1.}
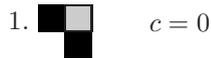
\begin{picture}(25,25)
\gs{10}{10}{1}
\bs{0}{10}{2}
\put(10,10){\line(1,0){10}} \put(10,20){\line(1,0){10}}
\put(10,10){\line(0,1){10}} \put(20,10){\line(0,1){10}}

\end{picture}
\quad \raisebox{10pt}{ $c=0$ }


\noindent \raisebox{21pt}{2a.}
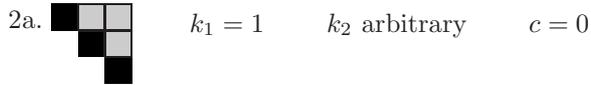
\begin{picture}(35,35)

\gs{20}{20}{2}
\gs{20}{10}{1}
\bs{0}{20}{3}
\put(30,30){\line(-1,0){20}} \put(30,20){\line(-1,0){20}}
\put(30,10){\line(-1,0){10}} \put(30,30){\line(0,-1){20}}
\put(20,30){\line(0,-1){20}} \put(10,30){\line(0,-1){10}}
\end{picture}
\quad \raisebox{30pt}{
\parbox[t]{5cm}{

\begin{tabbing}

$k_1=1$ \qquad \= $k_2$ arbitrary \qquad    \= $c=0$  \\

\end{tabbing}

}}

\noindent \raisebox{31pt}{2b.}
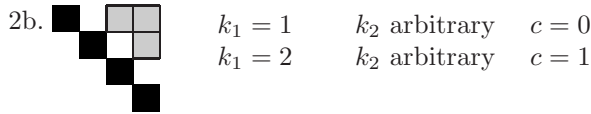
\begin{picture}(45,45)
\gs{30}{30}{2}
\gs{30}{20}{1}
\bs{0}{30}{4}
\put(40,40){\line(-1,0){20}} \put(40,30){\line(-1,0){20}}
\put(40,20){\line(-1,0){10}} \put(40,40){\line(0,-1){20}}
\put(30,40){\line(0,-1){20}} \put(20,40){\line(0,-1){10}}

\end{picture}
\quad \raisebox{40pt}{
\parbox[t]{5cm}{

\begin{tabbing}
$k_1=1$ \qquad \= $k_2$ arbitrary \quad \=  $c=0$ \\
$k_1=2$ \qquad \> $k_2$ arbitrary   \> $c=1$ \\
\end{tabbing}

}}


\noindent \raisebox{31pt}{3a.}
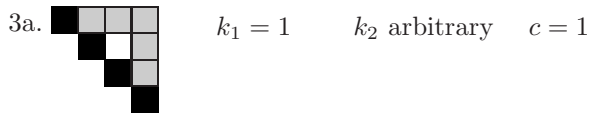
\begin{picture}(45,45)
\gs{30}{30}{3}
\gs{30}{20}{1}
\gs{30}{10}{1}
\bs{0}{30}{4}
\put(40,40){\line(-1,0){30}} \put(40,30){\line(-1,0){30}}
\put(40,20){\line(-1,0){10}} \put(40,10){\line(-1,0){10}}
\put(40,40){\line(0,-1){30}} \put(30,40){\line(0,-1){30}}
\put(20,40){\line(0,-1){10}} \put(10,40){\line(0,-1){10}}

\end{picture}
\quad \raisebox{40pt}{
\parbox[t]{5cm}{
%
\begin{tabbing}
$k_1=1$ \qquad \= $k_2$ arbitrary  \quad \=  $c=1$ \\
\end{tabbing}

}}

\noindent \raisebox{41pt}{3b.}
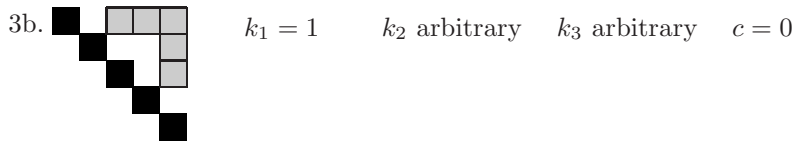
\begin{picture}(55,55)
\gs{40}{40}{3}
\gs{40}{30}{1}
\gs{40}{20}{1}
\bs{0}{40}{5}
\put(50,50){\line(-1,0){30}} \put(50,40){\line(-1,0){30}}
\put(50,30){\line(-1,0){10}} \put(50,20){\line(-1,0){10}}
\put(50,50){\line(0,-1){30}} \put(40,50){\line(0,-1){30}}
\put(30,50){\line(0,-1){10}} \put(20,50){\line(0,-1){10}}

\end{picture}
\quad \raisebox{50pt}{
\parbox[t]{5cm}{
%
%
\begin{tabbing}
$k_1=1$ \qquad \=  $k_2$ arbitrary \quad  \= $k_3$ arbitrary \quad \= $c=0$  \\
\end{tabbing}

}}


If there are at least $4$ blocks in the first row, then by Lemma~\ref{in_one_row_3_4} their height cannot be greater than one.

Suppose nonzero blocks stand only in the first row and in the last column (denote the number of blocks in the first row by $m$) and suppose that their height is $k_1 = 1$. Denote the complexity for this case by $c_{m,a}$ if $r = m + 1$ and by $c_{m,b}$ if $r = m + 2$ (here $r$ is the number of diagonal blocks). Then we have the following lemma.

\begin{lemma}
Suppose $m \geqslant 3$. We have $c_{m,a}=c_{m-1,b}+1$, $c_{m,b}=c_{m-1,a}$.
\label{so_k1=1}
\end{lemma}

The proof is similar to the proof of Lemma~\ref{so_k1=1}.

In Cases $4$ and $5$ we can easily compute complexities using this lemma, and $c \geqslant 2$ for $m \geqslant 6$.


\noindent \raisebox{41pt}{4a.}
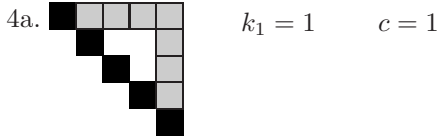
\begin{picture}(55,55)

\gs{40}{40}{4}
\gsdown{40}{40}{4}
\bs{0}{40}{5}

\put(50,50){\line(-1,0){40}} \put(50,40){\line(-1,0){40}}
\put(50,30){\line(-1,0){10}} \put(50,20){\line(-1,0){10}}
\put(50,10){\line(-1,0){10}} \put(50,50){\line(0,-1){40}}
\put(40,50){\line(0,-1){40}} \put(30,50){\line(0,-1){10}}
\put(20,50){\line(0,-1){10}} \put(10,50){\line(0,-1){10}}

\end{picture}
\quad \raisebox{50pt}{
\parbox[t]{5cm}{
%
\begin{tabbing}
$k_1=1$ \qquad \=  $c=1$  \\

\end{tabbing}

}}

\noindent \raisebox{51pt}{4b.}
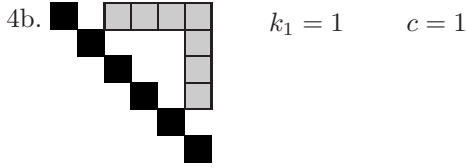
\begin{picture}(65,65)
\gs{50}{50}{4}
\gsdown{50}{50}{4}
\bs{0}{50}{6}

\put(60,60){\line(-1,0){40}} \put(60,50){\line(-1,0){40}}
\put(60,40){\line(-1,0){10}} \put(60,30){\line(-1,0){10}}
\put(60,20){\line(-1,0){10}} \put(60,60){\line(0,-1){40}}
\put(50,60){\line(0,-1){40}} \put(40,60){\line(0,-1){10}}
\put(30,60){\line(0,-1){10}} \put(20,60){\line(0,-1){10}}
\end{picture}
\quad \raisebox{60pt}{
\parbox[t]{5cm}{
%
%
\begin{tabbing}
$k_1=1$ \qquad \= $c=1$  \\
  \end{tabbing}

}}


\noindent \raisebox{51pt}{5a.}
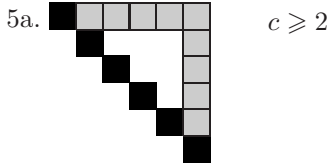
\begin{picture}(65,65)
\gs{50}{50}{5}
\gsdown{50}{50}{5}
\bs{0}{50}{6}

\put(60,60){\line(-1,0){50}} \put(60,50){\line(-1,0){50}}
\put(60,40){\line(-1,0){10}} \put(60,30){\line(-1,0){10}}
\put(60,20){\line(-1,0){10}} \put(60,10){\line(-1,0){10}}
\put(60,60){\line(0,-1){50}} \put(50,60){\line(0,-1){50}}
\put(40,60){\line(0,-1){10}} \put(30,60){\line(0,-1){10}}
\put(20,60){\line(0,-1){10}} \put(10,60){\line(0,-1){10}}

\end{picture}
\quad \raisebox{50pt}{
\parbox[t]{5cm}{
%
%
$c \geqslant 2$

}}

\noindent \raisebox{61pt}{5b.}
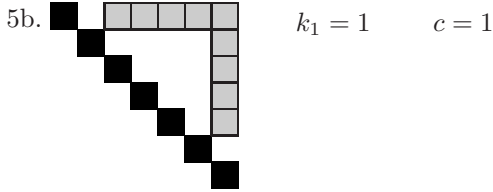
\begin{picture}(75,75)
\gs{60}{60}{5}
\gsdown{60}{60}{5}
\bs{0}{60}{7}

\put(70,70){\line(-1,0){50}} \put(70,60){\line(-1,0){50}}
\put(70,50){\line(-1,0){10}} \put(70,40){\line(-1,0){10}}
\put(70,30){\line(-1,0){10}} \put(70,20){\line(-1,0){10}}
\put(70,70){\line(0,-1){50}} \put(60,70){\line(0,-1){50}}
\put(50,70){\line(0,-1){10}} \put(40,70){\line(0,-1){10}}
\put(30,70){\line(0,-1){10}} \put(20,70){\line(0,-1){10}}

\end{picture}
\quad \raisebox{70pt}{
\parbox[t]{5cm}{
%
%
\begin{tabbing}
$k_1=1$ \qquad \=  $c=1$  \\
\end{tabbing}

}}


\noindent \raisebox{31pt}{7.}
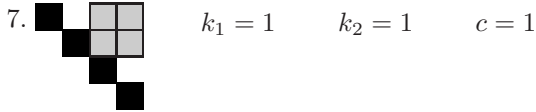
\begin{picture}(45,45)
\gs{30}{30}{2}
\gs{30}{20}{2}
\bs{0}{30}{4}
\put(40,40){\line(-1,0){20}} \put(40,30){\line(-1,0){20}}
\put(40,20){\line(-1,0){20}} \put(40,40){\line(0,-1){20}}
\put(30,40){\line(0,-1){20}} \put(20,40){\line(0,-1){20}}
\end{picture}
\quad \raisebox{40pt}{
\parbox[t]{5cm}{

\begin{tabbing}
$k_1=1$ \qquad \= $k_2=1$ \qquad \= $c=1$  \\
\end{tabbing}

}}


\noindent \raisebox{31pt}{8.}
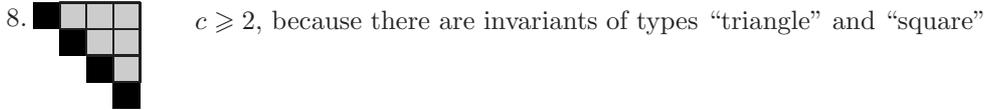
\begin{picture}(45,45)

\gs{30}{30}{3}
\gs{30}{20}{2}
\gs{30}{10}{1}
\bs{0}{30}{4}
\put(40,40){\line(-1,0){30}} \put(40,30){\line(-1,0){30}}
\put(40,20){\line(-1,0){20}} \put(40,10){\line(-1,0){10}}
\put(40,40){\line(0,-1){30}} \put(30,40){\line(0,-1){30}}
\put(20,40){\line(0,-1){20}} \put(10,40){\line(0,-1){10}}

\end{picture}
\quad \raisebox{30pt}{
\parbox[t]{12cm}{
%
%
$c \geqslant 2$, because there are invariants of types ``triangle'' and ``square''
}}


\noindent If we add  blocks to the first row (and add blocks symmetrical to them), then these cases do not appear.

\noindent If there are $\geqslant 3$ blocks in the second row, then $c \geqslant 2$, because there are $2$ invariants of type ``square''.

\medskip

Combining all cases together we obtain the classification given in Table~\ref{sp}.

\section {Case of exceptional groups}

Fix a Borel subgroup $B \subseteq G$ and a maximal torus  $T \subseteq B$. Suppose $\Delta$ is the system of roots with respect to $T$, $\Pi$ is the system of simple roots corresponding to the choice of $B$, $I \subseteq \Pi$ is a subset. Any parabolic subgroup containing $B$ coincides with a standard parabolic subgroup $P_I$ whose Lie algebra can be decomposed into the direct sum of the Lie algebra of $T$ and the root subspaces corresponding to positive roots and roots that are linear combinations with integer coefficients of roots from $I$, i.e.,
$$ \mathfrak{p}_I = \mathfrak{t} \oplus
      \bigoplus_{\{\alpha > 0\} \cup \{\alpha \in \mathbb{Z}I\}} \mathfrak{g}_{\alpha}$$
The Lie algebra $\mathfrak{p}_I$ can be decomposed into the direst sum of the standard  Levi subalgebra $\mathfrak{l}$ and the Lie algebra of the unipotent radical. The roots from $\mathbb{Z}I$ correspond to the Lie subalgebra $\mathfrak{l}$ and the other roots $\{\alpha > 0\} \cap \{\alpha \notin \mathbb{Z}I\}$ correspond to the unipotent radical.

Suppose $P=P_I=L \rightthreetimes P_u$ and $Q=P_J=M \rightthreetimes Q_u$ are two parabolic subgroups. Then we have
\begin{equation*}
\begin{split}
\mathfrak{l} \cap \mathfrak{m} &= \mathfrak{t}
    \oplus \bigoplus_{ \alpha \in \mathbb{Z}(I \cap J)} \mathfrak{g}_{\alpha}, \\
\mathfrak{p}_u \cap \mathfrak{q}_u &= \bigoplus_{\alpha > 0, \alpha \notin \mathbb{Z}I \cup \mathbb{Z}J}
   \mathfrak{g}_{\alpha}.
\end{split}
\end{equation*}

\medskip

Now we describe a general method of computing the complexity of a linear representation of a reductive group \cite[section~1.4]{count_complexity}. Suppose $G$ is a reductive group, $V$ is its linear representation.  Denote by $v_{\lambda}$ a vector of weight $\lambda$. Consider a lowest weight vector $v_{-\lambda^{*}}$ of $V$. We can decompose the space $V$ as follows: $V = \langle v_{-\lambda^{*}} \rangle \oplus W$, where $W$ is $B$-stable. Consider an open $B$-stable subset $\mathring{V} = \mathbb{C^{\times}} v_{-\lambda^{*}} \oplus W$ in $V$. Let $P$ be the parabolic subgroup preserving the line $\langle v_{\lambda^{*}} \rangle \subseteq V^{\ast}$. Decompose $P$ into a semidirect product of the Levi subgroup and the unipotent radical: $P=L \rightthreetimes P_u$. Let $V'$ be an $L$-stable complementary subspace to $\mathfrak{p}_u v_{-\lambda^{*}}$ in $W$, i.e., $\mathfrak{p}_u v_{-\lambda^{*}} \oplus V' = W$. The subset $\mathring{V}$ is isomorphic to the direct product $\mathring{V} =P_u \times (\mathbb{C^{\times}} v_{-\lambda^{*}} \oplus V')$ as a $B$-variety. On the right-hand side, $P_u$ acts on the first factor by left translations, and  $B \cap L$ acts on the first factor by conjugation, while the action on the second factor is induced from the action of $L$. Therefore the codimension of a general orbit for the action $B= (B \cap L) \rightthreetimes P_u  : \mathring{V}$ equals the codimension of a general orbit for the action $B \cap L : \langle v_{-\lambda^{*}} \rangle \oplus V'$. Thus we reduced the question about the complexity of the action $G:V$ to the complexity of a smaller group $L$ acting  on a smaller space $\langle v_{-\lambda^{\ast}} \rangle \oplus V'$. In this way we can construct a sequence of groups $L^{(i)}$ and spaces $V^{(i)}$:

$$G=L^{(0)} \supseteq L^{(1)} \supseteq \dots \supseteq L^{(s)}$$

$$V=V^{(0)} \supseteq V^{(1)} \supseteq \dots \supseteq V^{(s)}$$

\noindent such that all irreducible  $L^{(s)}$-submodules $V^{(s)}$ are one-dimensional. Then the action $L^{(s)}$ on $V^{(s)}$ is determined by the weights $\mu_1, \dots ,\mu_N$ and the complexity equals $\dim V^{(s)} - \rk \langle \mu_1, \dots ,\mu_N \rangle = N - \rk \langle \mu_1, \dots ,\mu_N \rangle $.

Now we explain how this method is applied in our case. The intersection of the Levi subgroups $L \cap M$ and the intersection of the Lie algebras of the unipotent radicals $\mathfrak{p}_u \cap \mathfrak{q}_u$ are determined by some subsets of roots $E_1$ and $F_1$ corresponding to the weight subspaces of Lie algebras with nonzero weights. Suppose $\mu_1$ is a minimal root from $F_1$.
Let $E_1' = \{\alpha \in E_1 \mid \alpha + \mu_1 \in F_1\}$, $F_1' = \{\alpha + \mu_1 \mid \alpha \in E_1'\}$.
Put $E_2 = E_1 \setminus E_1'$, $F_2 = F_1 \setminus (F_1' \cup \{\mu_1\})$.
In the same way we construct $\mu_2$ and $E_3, F_3$ for $E_2$ and $F_2$ and so on, while $F_i$ is nonempty. So we obtain a set of weights $\mu_1, \mu_2, \dots ,\mu_N$ and complexity equals  $N - \rk \langle \mu_1, \dots, \mu_N \rangle$.

For every exceptional group $G$ we first compute complexity of double flag varieties for maximal parabolic subgroups (they correspond to subsets of simple roots obtained from the set of simple roots by removing one root). Then we reduce the parabolic subgroups. It is clear that we do not need to compute complexity for the cases where we have an estimate $c \geqslant 2$ from Lemmas \ref{subgroup}, \ref{complexity}.

Now consider particular groups.

\subsection{$\mathbf{G_2}$, $\mathbf{F_4}$}

For these groups there are no pairs of maximal parabolic subgroups for which the estimate on complexity is $\leqslant 1$.

\subsection {$\mathbf{E_8}$}

For all pairs of parabolic subgroups except the pair corresponding to the pair of roots $(\alpha_1, \alpha_1)$ we have an estimate on complexity $c \geqslant 2$. For this pair the complexity equals $2$. So there are no suitable cases.

\subsection {$\mathbf{E_6}$, $\mathbf{E_7}$}

The list of pairs of roots for which the estimate on complexity for corresponding parabolics is not greater than~$1$, and complexities are given in Table~\ref{e6e7}.

\begin{table}[!h]
\begin{center}
\begin{tabular}{|ll|c|}
\hline

\multicolumn{3}{|c|}{$E_6$} \\ \hline

$\Pi \backslash I$ & $\Pi \backslash J$ & complexity \\
\hline

1 & 1 & 0 \\
1 & 2 & 0 \\
1 & 4 & 0 \\
1 & 5 & 0 \\
1 & 6 & 0 \\
2 & 5 & 0 \\
4 & 5 & 0 \\
5 & 5 & 0 \\
5 & 6 & 0 \\
6 & 6 & 2 \\

1 & 1, 2 & 1 \\
1 & 1, 5 & 0 \\
1 & 1, 6 & 1 \\
1 & 4, 5 & 1 \\
1 & 5, 6 & 1 \\
5 & 1, 2 & 1 \\
5 & 1, 5 & 0 \\
5 & 1, 6 & 1 \\
5 & 4, 5 & 1 \\
5 & 5, 6 & 1 \\
\hline

\end{tabular}
\begin{tabular}{|ll|c|}
\hline

\multicolumn{3}{|c|}{$E_7$} \\ \hline

$\Pi \backslash I$ & $\Pi \backslash J$ & complexity \\ \hline

1 & 1 & 0 \\
1 & 2 & 1 \\
1 & 6 & 0 \\
1 & 7 & 0 \\
6 & 6 & 2 \\

1 & 1, 2 & 2 \\
1 & 1, 6 & 2 \\
\hline
\end{tabular}
\caption{pairs of parabolic subgroups such that the estimate on complexity is $\leqslant 1$, and corresponding complexities}
\label{e6e7}
\end{center}
\end{table}

\bigskip

For example we consider in details the case for the group $E_6$ and subsets $\Pi \backslash I = \{\alpha_1\}$, $\Pi \backslash J = \{\alpha_5\}$. In this case, in the notation of \cite[Table $1$]{Vin}, $E_1 = \{\ee{i}{j} \mid i,j = 2, \dots, 5; i < j\} \cup \{\eeee{6}{i}{j}{} \mid i,j = 2, \dots, 5; i < j\}$, $F_1 = \{\ee{1}{6} \} \cup \{\eeee{1}{i}{j}{} \mid i,j = 2, \dots, 5; i < j\} \cup \{\e\}$. Take $\ee{1}{6}$ as $\mu_1$. Then  $E_1' = \{\eeee{6}{i}{j}{} \mid i,j = 2, \dots, 5; i < j\}$ and
$F_1' = \{\eeee{1}{i}{j}{} \mid i,j = 2, \dots, 5; i < j\}$, whence we have $E_2 = \{\ee{i}{j} \mid i,j = 2, \dots, 5; i < j\}$, and $F_2 = \{\e\}$. Take $\e$ as $\mu_2$, then $F_3 = \varnothing$. The complexity equals
$2 - \rk \langle \mu_1, \mu_2 \rangle = 0$.

The final result is formulated in Theorem  \ref{special}.


\begin{thebibliography} {99}

\bibitem{Lit} Littelmann P. \textit{On spherical double cones, J.Algebra}, \textbf{66}
(1994), 142-157.

\bibitem{St} Stembridge J. \textit {Multiplicity-free products and
restrictions of Weyl characters}, Represent. Theory 7 (2003),
404--439.

\bibitem{Pan} Panyushev D.I. \textit{Complexity and rank of double cones and
tensor product decompositions}, Comment. Math. Helv. \textbf{68}
(1993), no. 3, 455-468.

\bibitem{Vin} Onishchik A.L., Vinberg E.B.,  \textit{Lie groups and algebraic groups}, Springer, Berlin-Heidelberg-New York, 1990.

\bibitem{KKLV} Knop F., Kraft H., Luna D., Vust Th. \textit{Local properties
of algebraic group actions}, Algebraische Transformationsgruppen und Invariantentheorie
(Kraft H., Slodowy P., Springer T.A., eds.), DMV seminar, vol. 13, pp. 63-75, Birkhauser,
Basel-Boston-Berlin, 1989.

\bibitem{FMSS} Fulton W., MacPherson R., Sottile F., Sturmfles B. \textit{Intersection theory on spherical varieties},
J. Algebraic Geom. \textbf{4} (1995), no. 1, 181-193.

\bibitem{Bri} Brion M. \textit{Groupe de Picard et nombres caract\'{e}ristiques
des vari\'{e}t\'{e}s sph\'{e}riques}, Duke Math. J. \textbf{58} (1989), 397-424.

\bibitem{Tim} Timashev D.A. \textit{Cartier divisors and geometry of normal $G$-varieties},
Transformation Groups \textbf{5} (2000), no.2, 181-204.

\bibitem{complexity_subvariety} Vinberg E.B. \textit{Complexity of actions of reductive groups}, Funct. Anal. Appl. \textbf{20} (1986), no. 1, 1-11.


\bibitem{count_complexity} Panyushev D.I. \textit{Complexity and rank of actions in Invariant theory}, J.Math.Sci. \textbf{95} (1999), no. 1, 1925-1985.

\end{thebibliography}
\end{document}